\documentclass{article}
\usepackage{latexsym,amsfonts,amsthm,amsmath,amscd,amssymb}
\usepackage[dvips]{graphicx}

\newtheorem{lemma}{Lemma}[section]
\newtheorem{thm}[lemma]{Theorem}

\newtheorem{prop}[lemma]{Proposition}
\newtheorem{cor}[lemma]{Corollary}

\newtheorem{oss}[lemma]{Observation}
\newtheorem{example}[lemma]{Example}
\newtheorem{defn}[lemma]{Definition}

\newcommand{\cl}{C \kern -0.1em \ell}

\newcommand{\finedimo}{{\hfill\hbox{$\square$}\vspace{2pt}}}

\newcommand\matR{{\mathbb{R}}}
\newcommand\matN{{\mathbb{N}}}

\renewcommand{\hbar}{{\overline{h}}}

\newfont{\Got}{eufm10 scaled 1200}

\newcommand\calD{{\mathcal D}}

\begin{document}

\title{Linear and multilinear algebra on diffeological vector spaces}

\author{Ekaterina~{\textsc Pervova}}

\maketitle

\begin{abstract}
\noindent This paper deals with some basic constructions of linear and multilinear algebra on finite-dimensional diffeological vector spaces. After verifying
that some standard constructions carry over to the diffeological setting, we consider the diffeological dual formally checking that the assignment to each
space of its dual defines a covariant functor from the category of finite-dimensional diffeological vector spaces to the category of standard (that is, carrying
the usual smooth structure) vector spaces. We verify that the diffeological tensor product enjoys the typical properties of the usual tensor product
(associativity, distributivity, commutativity with taking), after which we focus on the so-called smooth direct sum decompositions, a phenomenon exclusive to
the diffeological setting. We then consider the so-called pseudo-metrics (diffeological analogues of scalar products). After recalling that a pseudo-metric
determines the natural pairing map onto the diffeological dual, which admits a well-defined left inverse for any fixed choice of a basis, we discuss
decompositions of vector spaces into a smooth direct sum of the maximal isotropic subspace and a characteristic subspace; we show that, although such a
decomposition can be described on the basis of a pseudo-metric, it actually depends only on the choice of the coordinate system. Furthermore, we show that,
contrary to what was erroneously claimed (by me) elsewhere, a characteristic subspace is not unique and is not invariant under the diffeomorphisms of the space
on itself. The maximal isotropic subspace is on the other hand an invariant of the space itself, and this fact allows to assign to each its well-defined
characteristic quotient, obtaining another functor, this time a contravariant one, to the category of standard vector spaces. After discussing the diffeological
analogues of isometries, we end with some remarks concerning diffeological algebras and diffeological Clifford algebras. Perhaps a larger than usual part of
the paper recalls statements that already appear elsewhere, but when this is the case, we try to accompany them with new proofs and examples.


\end{abstract}

\section*{Introduction}

Already some basic multilinear (and even linear) algebra when considered for \emph{diffeological vector spaces} gives rise to some phenomena that have no
counterparts in the standard context. One curious observation that appears already at the most elementary level is that even for finite-dimensional
diffeological vector spaces not all (multi)linear maps between them are smooth. Thus, it may, or may not, be obvious whether the classical isomorphisms of
multilinear algebra continue to exist in the diffeological context, in the sense whether their restrictions to the smooth subspaces are well-defined and,
if so, whether these restrictions are diffeomorphisms in their turn. We start with considering these kinds of questions, in addition to providing a few explicit
proofs to the statements announced or implicit in\cite{wu}.

The paper might be a bit boring up to Section 1.5 (which is not to claim that it is not boring from that point on), since most often than not, it just says
that stuff works just like you are used to it does (except for ``linear'' not meaning automatically ``smooth''); the only statement that may not be immediately
obvious is the fact that (in the finite-dimensional case) the diffeological dual is always a \emph{standard} space (the statement itself appears already in
\cite{me2018}; we give here a simple and straightforward proof).

I had originally intended to write this paper to collect in one place a number of statements of linear/multilinear algebra scattered across my other papers
(where they appeared as needed for the problem at hand), accompanying them with more examples and perhaps filling in some gaps that would have been natural to
fill in; and this is more or less how Sections 1.5 and 1.6 were written. They deal with so-called smooth direct sum decompositions of vector spaces, which in
general are only a subset of all decompositions; roughly speaking, for a smooth direct sum, not only the space itself decomposes, and also its diffeological
(smooth) structure does in some sense.

Then, in Section 2 which deals with pseudo-metrics (the diffeological version of scalar product) new phenomena emerged (that I had previously missed). The most
important of them, which, although it appears in the context of pseudo-metrics, can be defined independently, is a semi-canonical decomposition of each
finite-dimensional diffeological vector space into a smooth direct sum of its \emph{characteristic} subspace and its \emph{maximal isotropic} subspace, which
are in some sense the largest standard and the largest ``wholly nonstandard'' parts. This decomposition depends on the choice of the basis, and --- contrary to
what was affirmed in \cite{me2018} --- in particular the characteristic subspace is not uniquely defined by the space only, nor is it true that there can be
only maximal subspace that is standard and splits off smoothly (implying that certain statements that depended on that erroneous statement, such as some in
\cite{me2017b}, need some extra assumptions), as Examples \ref{characteristic-not-unique-ex}, \ref{more-than-one-splitting-ex}, and
\ref{characteristic-not-invt-under-diffeo-ex} show. On the hand, the subspace that \emph{is} uniquely defined is the maximal isotropic subspace (which is indeed
the maximal isotropic subspace of any pseudo-metric on the given space and can also be defined as the intersection of kernels of all smooth linear functions on
the given space), and this in particular allows to assign to each space the corresponding \emph{characteristic quotient} (such being standard spaces) to which
all smooth linear maps descend.

The last section deals with diffeological algebras, up to Clifford algebras associated to pseudo-metrics. This has less of a new material; still, we observe
that a Clifford algebra of form $\cl(V,g)$ is essentially determined by the corresponding Clifford algebra associated to its characteristic subspace (assuming
that we make a fixed choice of one), which, since it is naturally contained in $\cl(V,g)$ as a vector subspace, splits off smoothly in it, and provide more
illustrative examples.

Recent results particularly relevant for infinite-dimensional diffeological vector spaces appear in \cite{christensen-wu2017}. The main (and most easily 
accessible) reference for the subject of diffeology in general, and diffeological vector spaces in particular, is the excellent and comprehensive source 
\cite{iglesiasBook}. Our work also builds to some extent on the definitions and facts already presented in \cite{wu}, and previously in \cite{vincent} (see also 
\cite{iglesiasVector}). Such basic notions as those of the diffeological dual and the tensor product,were announced, in a definitive manner, in \cite{wu}, where 
the discussion is rather concise; part of our intention is to render explicit what is implicit there, and to provide specific motivations stemming from various 
examples. I've also cited some of my own works (\cite{me2017b} --- \cite{me2019}), sometimes to render the present text more comprehensive (and to better 
illustrate them), sometimes to remediate to some errors that creeped in into them.

\paragraph{Acknowledgments} I would like to thank Dan Christensen and Enxin Wu for their comments on the first section of this paper.

\section{Diffeological spaces and diffeological vector spaces}

We spell out the diffeological counterparts of some standard multilinear algebra notions and facts.

\subsection{Main definitions}

Here we recall as briefly as possible the main notions regarding diffeology (\cite{So1}, \cite{So2}) and diffeological (vector) spaces; more details can be
found in \cite{iglesiasBook}; see \cite{iglesiasVector}, \cite{vincent}, \cite{wu} for diffeological vector spaces.

\paragraph{Diffeological spaces} A \textbf{diffeological space} (see \cite{So2}) is a pair $(X,\calD_X)$ where $X$ is a set and $\calD_X$ is a specified
collection of maps $U\to X$ (called \textbf{plots}) for each open set $U$ in $\matR^n$ and for each $n\in\matN$, such that for all open subsets
$U\subseteq\matR^n$ and $V\subseteq\matR^m$ the following three conditions are satisfied:
\begin{enumerate}
\item (The covering condition) Every constant map $U\to X$ is a plot;
\item (The smooth compatibility condition) If $U\to X$ is a plot and $V\to U$ is a smooth map (in the usual sense) then the composition $V\to U\to X$ is also
a plot;
\item (The sheaf condition) If $U=\cup_iU_i$ is an open cover and $U\to X$ is a set map such that each restriction $U_i\to X$ is a plot then the entire map
$U\to X$ is a plot as well.
\end{enumerate}
Usually, one just writes $X$ to denote a diffeological space; a standard example of a diffeological space is a smooth manifold $M^n$, with the diffeology given
by all smooth maps of form $U\to M^n$, for $U$ a domain in some $\matR^k$. If we have two diffeological spaces, $X$ and $Y$, and a set map $f:X\to Y$ between
them, this map is said to be \textbf{smooth} if for every plot $p:U\to X$ of $X$ the composition $f\circ p$ is a plot of $Y$.

\paragraph{Comparing diffeologies} Given a set $X$, the set of all possible diffeologies on $X$ is partially ordered by inclusion: a diffeology $\calD$ on $X$
is said to be \textbf{finer} than another diffeology $\calD'$ if $\calD\subset\calD'$ (whereas $\calD'$ is said to be \textbf{coarser} than $\calD$).

\paragraph{Generated diffeology and quotient diffeology} These are two (out of many) ways to construct a diffeology. If $X$ is a set and we are given a set of
maps $A=\{U_i\to X\}_{i\in I}$, the \textbf{diffeology generated by $A$} is the smallest, with respect to inclusion, diffeology on $X$ that contains $A$; its
plots are either locally constant or locally factor through those of $A$.

If now $X$ is a diffeological space, let $\sim$ be an equivalence relation on $X$, and let $\pi:X\to Y:=X/\sim$ be the quotient map. The \textbf{quotient
diffeology} (\cite{iglesiasBook}) on $Y$ is the diffeology in which $p:U\to Y$ is the diffeology in which $p:U\to Y$ is a plot if and only if each point in $U$
has a neighbourhood $V\subset U$ and a plot $\tilde{p}:V\to X$ such that $p|_{V}=\pi\circ\tilde{p}$.

\paragraph{Subset diffeology} Let $X$ be a diffeological space, and let $Y\subseteq X$ be its subset. The \textbf{subset diffeology} on $Y$ is the coarsest
diffeology on $Y$ making the inclusion map $Y\hookrightarrow X$ smooth.

\paragraph{Pushforwards and pullbacks of a diffeology} For any diffeological space $X$, any set $X'$, and any map $f:X\to X'$ there exists a finest diffeology
on $X'$ that makes the map $f$ smooth; it is called the \textbf{pushforward of the diffeology of $X$ by the map $f$}. If now we have a map $f:X'\to X$ then the
\textbf{pullback} of the diffeology of $X$ by the map $f$ is the coarsest diffeology on $X'$ such that $f$ is smooth.

\paragraph{The diffeological direct product} Let $\{X_i\}_{i\in I}$ be a collection of diffeological spaces. The \textbf{product diffeology} on the direct
product $X=\prod_{i\in I}X_i$ is the coarsest diffeology such that for each index $i\in I$ the natural projection $\pi_i:\prod_{i\in I}X_i\to X_i$ is smooth.

\paragraph{Functional diffeology} Let $X$, $Y$ be two diffeological spaces, and let $C^{\infty}(X,Y)$ be the set of smooth maps from $X$ to $Y$. Let \textsc{ev}
be the \emph{evaluation map}, defined by
$$\mbox{\textsc{ev}}:C^{\infty}(X,Y)\times X\to Y\mbox{ and }\mbox{\textsc{ev}}(f,x)=f(x).$$ The \textbf{functional diffeology} on $C^{\infty}(X,Y)$ is the
coarsest diffeology such that the evaluation map is smooth.

\paragraph{Diffeological vector spaces} Let $V$ be a vector space over $\matR$. The \textbf{vector space diffeology} on $V$ is any diffeology of $V$ such that
the addition and the scalar multiplication are smooth, that is,
$$[(u,v)\mapsto u+v]\in C^{\infty}(V\times V,V)\mbox{ and }[(\lambda,v)\mapsto\lambda v]\in C^{\infty}(\matR\times V,V),$$ where $V\times V$ and $\matR\times V$
are equipped with the product diffeology; equipped with a vector space diffeology, $V$ is called a \textbf{diffeological vector space}.

\paragraph{Smooth linear maps, subspaces and quotients} Given two diffeological vector spaces $V$ and $W$, the space of \textbf{smooth linear maps} between them
is denoted by $L^{\infty}(V,W)=L(V,W)\cap C^{\infty}(V,W)$. This is an $\matR$-linear subspace of $L(V,W)$, frequently a proper subspace. A \textbf{subspace} of
a diffeological vector space $V$ is any vector subspace of $V$ endowed with the subset diffeology (and is a diffeological vector space on its own). Finally, if
$V$ is a diffeological vector space and $W\leqslant V$ is a subspace of it then the quotient $V/W$ is a diffeological vector space with respect to the quotient
diffeology.

\paragraph{Direct sum of diffeological vector spaces} Let $V_1,\ldots,V_n$ be a collection of diffeological vector spaces. The usual direct sum
$V_1\oplus\ldots\oplus V_n$ of these spaces, equipped with the \emph{product} diffeology, is a diffeological vector space.

\paragraph{Fine diffeology on vector spaces} The \textbf{fine diffeology} on a vector space $V$ is the \emph{finest} vector space diffeology on it, and a vector
space carrying one is called a \emph{fine vector space}. Note that \emph{any} linear map between two fine vector spaces is smooth. In the finite-dimensional case,
the fine spaces are precisely the spaces (diffeomorphic to) $\matR^n$ with their standard diffeology (one consisting of all usually smooth maps into them).

\paragraph{Diffeological dual} For a diffeological vector space $V$, its \textbf{diffeological dual} $V^*$ (\cite{vincent}, \cite{wu}) is defined as the set of
all smooth linear maps $V\to\matR$ into the standard $\matR$, endowed with the functional diffeology.

\paragraph{The tensor product} Given a finite collection $V_1,\ldots,V_n$ of diffeological vector spaces, their usual tensor product
$V_1\otimes\ldots\otimes V_n$ is endowed with the \textbf{tensor product diffeology} (see \cite{vincent}, \cite{wu}) defined as the quotient diffeology of the
finest vector space diffeology on the free product $V_1\star\ldots\star V_n$ that contains the product diffeology on their direct product. The diffeological
tensor product possesses the usual universal property by Theorem 2.3.5 of \cite{vincent}, namely, if $W$ is another diffeological vector space then the space of
all smooth linear maps $V_1\otimes\ldots\otimes V_n\to W$, considered with the functional diffeology, is diffeomorphic to the space of all smooth multilinear
maps $V_1\times\ldots\times V_n\to W$ (also considered with the functional diffeology).

\subsection{Smooth linear and bilinear maps}

In this section we consider some of the classic isomorphisms of multilinear algebra, showing that they do extend into the context of diffeology, while providing
examples that show that the \emph{a priori} difference is significant.

\subsubsection{Linear maps and smooth linear maps}

The sometimes significant difference just mentioned is illustrated by the following example; note that the existence of such examples has already been mentioned
in \cite{wu}, see Example 3.11.

\begin{example}\label{linear-nonsmooth-example}
This is an example of $V$ such that $L^{\infty}(V,\matR)<L(V,\matR)$. Let $V=\matR^n$ equipped with the coarse diffeology; we claim that the only smooth linear
map $V\to\matR$ is the zero map. Indeed, let $f:V\to\matR$ be a linear map; it is smooth if and only if, for any plot $p$ of $V$, the composition $f\circ p$ is
a plot of $\matR$, \emph{i.e.}, $f\circ p$ is a usual smooth map $U\to\matR$ for some domain $U\subseteq\matR^k$. However, by definition of the coarse
diffeology, $p$ is allowed to be \emph{any} set map $U\to V$, so it may not even be continuous.

To provide a specific instance, choose some basis $\{v_1,\ldots,v_n\}$ of $V=\matR^n$, and a basis $\{v\}$ of $\matR$.\footnote{Obviously, the respective
canonical bases would do the job just fine.} With respect to these, $f$ is given by $n$ real numbers, more precisely, by the matrix $(a_1\,\,\,\ldots\,\,\,a_n)$.
Consider the following $n$ plots $p_i$ of $V$ with $i=1,\ldots,n$, defined by setting $p_i:\matR\to V$ and $p_i(x)=|x|v_i$; then $(f\circ p_i)(x)=a_i|x|v$. The
only way for this latter map to be smooth is to have $a_i=0$, and this must hold for $i=1,\ldots,n$, whence our claim.
\end{example}

The example just given shows that the \emph{a priori} issue of there being diffeological vector spaces $V$, $W$ such that $L^{\infty}(V,W)<L(V,W)$ does indeed
occur, and rather easily. True, this requires some rather surprising vector spaces/diffeologies for this happen; but diffeology was designed for dealing with
surprising, or at least unusual from the Differential Geometry point of view, objects,\footnote{This is a story beautifully told in the Preface and Afterword to
the excellent book \cite{iglesiasBook}.} and furthermore, the very essence of what diffeology aims to add to the `standard' differential setting is the
flexibility of what can be called smooth. In particular, the fact that \emph{any} given map $\matR^k\supseteq U\to X$ can be a plot for some diffeology on a
given set $X$ (for instance, for the diffeology generated by this map) easily gives rise to some surprising spaces.

\begin{example}
Once again, consider $V=\matR^n$ and some basis $\{v_1,\ldots,v_n\}$ of $V$; endow it with the vector space diffeology generated by all smooth maps plus the map
$p_i$ already mentioned, that is, the map $p_i:\matR\to V$ acting by $p_i(x)=|x|v_i$. Let $v$ be a generator of $\matR$ (\emph{i.e.}, any non-zero vector). Using
the same reasoning as in Example \ref{linear-nonsmooth-example}, one can show that if $f:V\to\matR$ is linear and, with respect to the bases chosen has matrix
$(a_1\,\,\,\ldots\,\,\,a_n)$, then for it to be smooth we must have $a_i=0$; hence the (usual vector space) dimension $L^{\infty}(V,\matR)$ (\emph{i.e.}, that of
the diffeological dual of $V$) is at most $n-1$.\footnote{One can actually show that $f$ is smooth \emph{if and only if} $a_i=0$, and so
$\dim(L^{\infty}(V,\matR))$ is precisely $n-1$; we do not elaborate on this, since we mostly interested in showing that it canbe strictly smaller (by any
admissible value, as we see below).}

This reasoning can be further extended by choosing some natural number $1<k<n$ and a set of $k$ indices $1\leqslant i_1<i_2<\ldots<i_k\leqslant n$, and endowing
$V$ with the vector space diffeology generated by all smooth maps plus the set $\{p_{i_1},\ldots,p_{i_k}\}$. Arguing as above, we can easily conclude that
$\dim(L^{\infty}(V,\matR))$ is at most $n-k$.
\end{example}

The examples just cited show, in particular, that the diffeological dual (defined in \cite{vincent} and \cite{wu}) of a diffeological vector space can be much
different from the usual vector space dual. Given the importance of the isomorphism-by-duality in multilinear algebra's arguments, the implications of this
difference deserve to be considered.

\subsubsection{Bilinear maps and smooth bilinear maps}

In this section we consider the same issues as above in the case of bilinear maps: given two diffeological vector spaces, what is the difference between the set
of all bilinear maps on one of them with values in the other, and the set of all such bilinear maps that in addition are smooth?

\paragraph{Smooth bilinear maps} Let $V$, $W$ be two diffeological spaces. As usual, a $W$-valued bilinear map is a map $V\times V\to W$ linear in each argument;
it is considered to be, or not, smooth with respect to the product diffeology on $V\times V$ and the diffeology on $W$. We first illustrate that what happens for
linear maps does (expectedly) happen for bilinear maps, \emph{i.e.}, the set of smooth bilinear maps can be strictly smaller than that of bilinear maps.

Our notation is as follows. Given $V$, $W$ two diffeological vector spaces, let $B(V,W)$ be the set of bilinear maps on $V$ with values in $W$, and let
$B^{\infty}(V,W)$ be the set of those bilinear maps that are smooth with respect to the product diffeology on $V\times V$ and the given diffeology on $W$.

\begin{example}
The examples seen in the previous section provide readily the instances of $V$ and $W$ such that $B^{\infty}(V,W)$ is a proper subspace of $B(V,W)$. Indeed, let
us take $V=\matR^n$ equipped with the coarse diffeology, and let $W=\matR$ considered with the standard diffeology. It is easy to extend the reasoning of Example
\ref{linear-nonsmooth-example} to show that for these two spaces $B^{\infty}(V,W)=0$.

Once again, take a basis $\{v_1,\ldots,v_n\}$ of $V$ and a basis $\{w\}$ of $W$; let $f\in B^{\infty}(V,W)$. Then with respect to the bases chosen $f$ is defined
by the matrix $(a_{ij})_{n\times n}$, where $f(v_i,v_j)=a_{ij}w$. For each $i=1,\ldots,n$ consider the already-seen map $p_i:\matR\to V$ given by $p_i(x)=|x|v_i$;
this map is a plot of $V$ by definition of the coarse diffeology (that includes all set maps from domains of various $\matR^k$ to $V$). Now call $p_{ij}$ the
product map $p_{ij}:\matR\to V\times V$, \emph{i.e.} the map given by $p_{ij}(x)=(p_i(x),p_j(x))$; it is obviously a plot for the product diffeology on
$V\times V$. Putting everything together, we get that $(f\circ p_{ij})(x)=a_{ij}|x|w$; since this must be a plot of $\matR$, \emph{i.e.} smooth in the usual
sense, we get that $a_{ij}=0$. The indices $i,j$ being arbitrary, we conclude that the only way for $f$ to be smooth is for it be the zero map, whence the
conclusion.
\end{example}

The example just given stresses the importance of making a distinction between bilinear maps and smooth bilinear maps, showing that the two spaces can be
(\emph{a priori}) quite different, and motivates the next paragraph.

\paragraph{The function spaces $B^{\infty}(V,W)$ and $L^{\infty}(V,L^{\infty}(V,W))$} As is well-known, in the usual setting each bilinear map can be viewed as a
linear map $V\to L(V,W)$. In the diffeological context, since \emph{a priori} we might have $L^{\infty}(V,W)<L(V,W)$, we need to consider the question of whether
any \emph{smooth} bilinear map can be seen as a \emph{smooth} map $V\to L^{\infty}(V,W)$, where the latter is endowed with the functional diffeology. By the
properties of the diffeologies involved, the answer is in the affirmative, as we now show (a slightly more general statement appears also in the proof of
Theorem 3.8 of \cite{wu}).

\begin{lemma}
Let $V$, $W$ be two diffeological vector spaces, let $f:V\times V\to W$ be a bilinear map smooth with respect to the product diffeology on $V\times V$ and the
given diffeology on $W$, and let $G:V\to L^{\infty}(V,W)$ be a linear map that is smooth with respect to the given diffeology on $V$ and the functional
diffeology on $L^{\infty}(V,W)$. Then:
\begin{itemize}
\item for every $v\in V$ the linear map $F(v):V\to W$ given by $F(v)(v')=f(v,v')$ is smooth;
\item the bilinear map $g:V\times V\to W$ given by $g(v,v')=G(v)(v')$ is smooth.
\end{itemize}
\end{lemma}

\begin{proof}
Let us prove the first statement. Fix a $v\in V$; to show that $F(v)$ is smooth, we need to show that for every plot $p:U\to V$ the composition $F(v)\circ p$ is
a plot of $W$. Fixing an arbitrary plot $p:U\to V$ of $V$, we define $\tilde{p}:U\to V\times V$ by setting $\tilde{p}(x)=(v,p(x))$ for all $x\in V$; this is
indeed a plot for the product diffeology, since the projections on both factors are smooth: $\pi_1\circ\tilde{p}$ is a constant map in $V$, while
$\pi_2\circ\tilde{p}=p$, which is a plot by assumption. We thus obtain $F(v)\circ p=f\circ\tilde{p}$; the latter map is a plot of $W$ since $f$ is smooth; and
since $p$ is any, so is $F(v)$.

To prove the second statement, it suffices to observe that $g$ writes as the composition $g=\mbox{\textsc{ev}}\circ(G\times\mbox{Id}_V)$; the map $\mbox{Id}_V$
being obviously smooth, $G$ being smooth by assumption, their product being smooth by definition of the product diffeology, and, finally, the evaluation map
$\mbox{\textsc{ev}}$ being smooth by the definition of the functional diffeology, we get the conclusion.
\end{proof}

What the above lemma gives us are the following two maps:
\begin{itemize}
\item the map $\tilde{F}:B^{\infty}(V,W)\to L(V,L^{\infty}(V,W))$ that assigns to each $f\in B^{\infty}(V,W)$ the map $F$ of the lemma (\emph{i.e.}, the
specified map that to each $v\in V$ assigns the smooth linear map $F(v):V\to W$). Observe that $F$ now writes as $F=\tilde{F}(f)$ and that the following relation
holds: $f=\mbox{\textsc{ev}}\circ(F \times\mbox{Id}_V)$;
\item the map $\tilde{G}:L^{\infty}(V,L^{\infty}(V,W))\to B^{\infty}(V,W)$ that assigns to each $G\in L^{\infty}(V,L^{\infty}(V,W))$ the map
$g=\mbox{\textsc{ev}}\circ(G\times\mbox{Id}_V)$. This latter map now writes as $g=\tilde{G}(G)$.
\end{itemize}

Before going further, we cite the following statement, which we will use immediately afterwards:

\begin{prop}\label{criterio-funct-diff} \emph{(\cite{iglesiasBook}, 1.57)}
Let $X$, $Y$ be two diffeological spaces, and let $U$ be a domain of some $\matR^n$. A map $p:U\to C^{\infty}(X,Y)$ is a plot for the functional diffeology of
$C^{\infty}(X,Y)$ if and only if the induced map $U\times X\to Y$ acting by $(u,x)\mapsto p(u)(x)$ is smooth.
\end{prop}

We are now ready to prove the following lemma:

\begin{lemma}\label{diff-bilin-twicelin-lem}
The following statements hold:
\begin{enumerate}
\item The map $\tilde{F}$ takes values in $L^{\infty}(V,L^{\infty}(V,W))$; furthermore, it is smooth with respect to the functional diffeologies of
$B^{\infty}(V,W)$ and $L^{\infty}(V,L^{\infty}(V,W))$.
\item The map $\tilde{G}$ is smooth with respect to the functional diffeologies of $L^{\infty}(V,L^{\infty}(V,W))$ and $B^{\infty}(V,W)$.
\item The maps $\tilde{F}$ and $\tilde{G}$ are inverses of each other.
\end{enumerate}
\end{lemma}

\begin{proof}
Let us prove 1. We first prove that $F:V\to L^{\infty}(V,W)$ is smooth. Let $p:U\to V$ be an arbitrary plot of $V$; by Proposition \ref{criterio-funct-diff}, in
order to to show that $F\circ p$ is a plot for the functional diffeology on $L^{\infty}(V,W)$, we need to consider the induced map $U\times V\to W$ that acts by
the assignment $(u,v')\mapsto (F\circ p)(u)(v')=F(p(u))(v')=f(p(u),v')=f\circ(p\times\mbox{Id}_V)(u,v')$ and show that it is smooth. Since $p\times\mbox{Id}_V$
is obviously a plot for the product diffeology on $V\times V$ and $f$ is smooth, $f\circ(p\times\mbox{Id}_V)$ is a plot of $W$, so it is naturally smooth.

Let us now show that $\tilde{F}:B^{\infty}(V,W)\to L^{\infty}(V,L^{\infty}(V,W))$ is smooth; taking $p:U\to B^{\infty}(V,W)$ a plot of $B^{\infty}(V,W)$, we need
to show that $\tilde{F}\circ p$ is a plot of $L^{\infty}(V,L^{\infty}(V,W))$. To do this, we apply again Proposition \ref{criterio-funct-diff}: it suffices to
consider the map $U\times V\to L^{\infty}(V,W)$ acting by
$(u,v)\mapsto(\tilde{F}\circ p)(u)(v)=\tilde{F}(p(u))(v)=\mbox{\textsc{ev}}\circ((F\circ p)\times\mbox{Id}_V)(u,v)$. Having already established that $F$ is
smooth, we can now conclude that $\tilde{F}$ is smooth as well.

Let us now prove the second point, \emph{i.e.}, that $\tilde{G}:L^{\infty}(V,L^{\infty}(V,W))\to B^{\infty}(V,W)$ is smooth, \emph{i.e.}, taking an arbitrary
plot $p:U\to L^{\infty}(V,L^{\infty}(V,W))$, we need to show that $\tilde{G}\circ p$ is a plot of $B^{\infty}(V,W)$. Applying again Proposition
\ref{criterio-funct-diff}, we consider the map $U\times(V\times V)\to W$ defined by
$(u,(v,v'))\mapsto(\tilde{G}\circ p)(u)(v,v')=
(\mbox{\textsc{ev}}\circ(p(u)\times\mbox{Id}_V))(v,v')=(\mbox{\textsc{ev}}\circ(p\times\mbox{Id}_{V\times V}))(u,(v,v'))$, which allows us to conclude that the
map is smooth, and therefore $\tilde{G}\circ p$ is a plot of $B^{\infty}(V,W)$; whence the conclusion.

To conclude, we observe that the third point follows immediately from the definitions of the two maps.
\end{proof}

We now get the desired conclusion, which does mimic what happens in the usual linear algebra case:

\begin{thm}
Let $V$ and $W$ be two diffeological vector spaces, let $B^{\infty}(V,W)$ be the space of all smooth bilinear maps $V\times V\to W$ considered with the
functional diffeology, and let $L^{\infty}(V,L^{\infty}(V,W))$ be the space of all smooth linear maps $V\to L^{\infty}(V,W)$ endowed, it as well, with the
functional diffeology. Then the spaces $B^{\infty}(V,W)$ and $L^{\infty}(V,L^{\infty}(V,W))$ are diffeomorphic as diffeological vector spaces.
\end{thm}

\begin{proof}
The desired diffeomorphism as diffeological spaces is given by the maps $\tilde{F}$ and $\tilde{G}$ of Lemma \ref{diff-bilin-twicelin-lem}. It remains to note
that these two maps are also linear (actually, as vector spaces maps they coincide with the usual constructions), and that all the functional diffeologies
involved are vector space diffeologies.
\end{proof}

\subsection{The diffeological dual}

We now comment on the already-standard notion of the diffeological dual $V^*=L^{\infty}(V,\matR)$ (\cite{vincent}, \cite{wu}). Note that, as has already been
observed in \cite{wu}, even in the finite-dimensional case the diffeological dual might be much smaller than the usual one, which corresponds precisely to the
case $L^{\infty}(V,\matR)<L(V,\matR)$ (also illustrated by our Example \ref{linear-nonsmooth-example}). A simple but natural question to ask at this point is,
suppose that $V$ is a finite-dimensional diffeological vector space such that $L^{\infty}(V,\matR)=L(V,\matR)$ as vector spaces; does this imply that $V$ is also
diffeomorphic to $V^*=L^{\infty}(V,\matR)$? The following proposition provides a positive answer to this question.\footnote{I would like to thank the anonymous
referee for pointing out a much simpler proof of this statement.}

\begin{prop}\label{diffeo-duality-prop}
Let $V$ be a finite-dimensional diffeological vector space such that $L^{\infty}(V,\matR)=L(V,\matR)$, \emph{i.e.}, such that every real-valued linear map from
$V$ is smooth. Then $V$ is diffeomorphic to $L^{\infty}(V,\matR)=L(V,\matR)$, \emph{i.e.}, to its diffeological dual.
\end{prop}

\begin{proof}
Let $\{v_1,\ldots,v_n\}$ be a basis of $V$, and let $v_i^*$ be the corresponding dual basis in the usual sense (\emph{i.e.}, if
$v=\alpha_1 v_1+\ldots+\alpha_n v_n$ then $v_i^*(v)=\alpha_i$). Observe that each $v_i^*$ is smooth by assumption.

Consider the linear map $v\mapsto\sum_{i=1}^n v_i^*(v)v_i^*$; this is a map from $V$ that takes values in $V^*$ by the already-made observation. Since $V$ is
finite-dimensional, this map is obviously bijective, with the inverse given by $v^*\mapsto v^*(v_i)v_i$. It remains to observe that this inverse is also smooth,
and so the map indicated is indeed a diffeomorphism of diffeological vector spaces $V$ and $V^*$.
\end{proof}

It is relatively easy to establish that the dual of a finite-dimensional diffeological vector space is always a standard space.

\begin{thm}\label{dual-is-standard-thm}
Let $V$ be a finite-dimensional diffeological vector space, and let $V^*$ be its diffeological dual. Then the functional diffeology on $V^*$ is standard.
\end{thm}

This was already proven in Theorem 3.11, \cite{me2018}; here we give a more direct proof.

\begin{proof}
Let $v_1,\ldots,v_n$ be any basis of $V$. Obviously, the linear maps $v^1,\ldots,v^n$ given by $v^i(v_j)=\delta_{ij}$ form a basis of a usual dual of $V$
(more precisely, of its underlying vector space), of which the diffeological dual $V^*$ is a subspace. Let $q:U\to V^*$ be a plot of $V^*$ for its functional
diffeology; write $q(u)=q_1(u)v^1+\ldots+q_n(u)v^n$, where each $q^i$ is the composition of $q$ with the projection of the usual dual onto (the span of) $v^i$.
Observe immediately that the evaluation of $q$ on any constant plot must be a smooth function; in particular, the evaluation of $q$ on the constant plot with
value $v_i$ must be smooth. Therefore $q_1,\ldots,q_n$ are all smooth functions, which means precisely that the diffeology of $V$ is standard.
\end{proof}

The dual being a standard space, all usual dual maps are automatically smooth.

\subsection{The tensor product}

In this section we discuss the definition (as given in \cite{wu}), and the relative properties, of the tensor product; we speak mostly of the case of two
factors, given that the extension to the case of more than two spaces is \emph{verbatim}.

\paragraph{The tensor product of maps} Let us consider two (smooth) linear maps between diffeological vector spaces, $f:V\to V'$ and $g:W\to W'$. As usual, we
have the tensor product map $f\otimes g:V\otimes W\to V'\otimes W'$, defined by $(f\otimes g)(\sum v_i\otimes w_i)=\sum f(v_i)\otimes g(w_i)$. We observe that
$f\otimes g$ is a smooth map (with respect to the tensor product diffeologies on $V\otimes W$ and $V'\otimes W'$) due to the properties of the product and the
quotient diffeologies.

\paragraph{The tensor product and the direct sum} Let $V_1$, $V_2$, $V_3$ be vector spaces; recall that in the usual linear algebra the tensor product is
distributive with respect to the direct sum, \emph{i.e.}: $$V_1\otimes(V_2\oplus V_3)\cong (V_1\otimes V_3)\oplus(V_2\otimes V_3),$$ via a canonical isomorphism,
which we denote by $T_{\otimes,\oplus}$. Now, if $V_1$, $V_2$, $V_3$ are diffeological vector spaces, then so are $V_1\otimes(V_2\oplus V_3)$ and $(V_1\otimes
V_3)\oplus(V_2\otimes V_3)$. It turns out, as already mentioned in \cite{wu}, Remark 3.9 (2), that the standard isomorphism between these spaces is also a
diffeomorphism. Below we provide an explicit proof of that statement.

\begin{lemma}\label{tensor-prod-is-distributive-lem} \emph{(\cite{wu})}
Let $V_1$, $V_2$, $V_3$ be diffeological vector spaces, and let $T_{\otimes,\oplus}:V_1\otimes(V_2\oplus V_3)\to(V_1\otimes V_3)\oplus(V_2\otimes V_3)$ be the
standard isomorphism. Then $T_{\otimes,\oplus}$ is smooth.
\end{lemma}

\begin{proof}
By the properties of the quotient diffeology, it is sufficient to show that the covering map
$\tilde{T}_{\ast,\oplus}:V_1\ast(V_2\oplus V_3)\to(V_1\ast V_3)\oplus(V_2\ast V_3)$ is smooth. Let $p:U\to V_1\ast(V_2\oplus V_3)$ be a plot; we must
show that $\tilde{T}_{\ast,\oplus}\circ p$ is a plot for $(V_1\ast V_3)\oplus(V_2\ast V_3)$. Let $\pi_1:V_1\ast(V_2\oplus V_3)\to V_1$ and
$\pi_{2,3}:V_1\ast(V_2\oplus V_3)\to(V_2\oplus V_3)$ be the natural projections; observe that by definition of the product diffeology, $\pi_{2,3}$ writes (at
least locally) as $\pi_{2,3}=(p_2,p_3)$, where $p_2$ is a plot of $V_2$ and $p_3$ is a plot of $V_3$.

Write now $\tilde{T}_{\ast,\oplus}\circ p$ as $\tilde{T}_{\ast,\oplus}\circ p=(p',p'')$; observe that $p'=(\pi_1\circ p,p_2)$, while $p''=(\pi_1\circ p,p_3)$.
These are plots for the sum diffeology on $(V_1\ast V_3)\oplus(V_2\ast V_3)$, hence the conclusion.
\end{proof}

\paragraph{The tensor product $V\otimes W$ as a function space} Recall that in the usual linear algebra context the tensor product of two finite-dimensional
vector spaces $V\otimes W$ is isomorphic to the spaces $L(V^*,W)$, the space of linear maps $V^*\to W$, and $L(W^*,V)$, the space of linear maps $W^*\to V$, via
isomorphisms given by:
\begin{itemize}
\item for $f\in V^*$, $v\in V$, and $w\in W$ we set $(v\otimes w)(f)=f(v)w$, extending by linearity;
\item for $g\in W^*$, $v\in V$, and $w\in W$ we set $(v\otimes w)(g)=g(w)v$, extending by linearity.
\end{itemize}

The question that we consider now is whether these isomorphisms continue to exist if all spaces we consider are finite-dimensional diffeological vector spaces,
all linear maps are smooth, and all function spaces are endowed with their functional diffeologies. The observations made regarding the frequently substantial
difference between a diffeological vector space $V$ and its diffeological dual $V^*$ suggest that we consider again one of our examples.

\begin{example}\label{tensor-not-functional-example}
Let $V=\matR^n$ for $n\geqslant 2$ with the coarse diffeology, and let $W=\matR$ with the standard diffeology. Then, as shown in Example
\ref{linear-nonsmooth-example}, the diffeological dual of $V$ is trivial: $V^*=\{0\}$; this obviously implies that $L^{\infty}(V^*,W)=\{0\}$. Recall also that,
the diffeology of $W$ being fine, its dual is isomorphic to $W$, so we have $W\cong W^*\cong\matR$; furthermore, as it occurs for all fine diffeological vector
spaces (see \cite{iglesiasBook}), we have $L^{\infty}(W^*,V)=L(W^*,V)\cong V$.

Since the total space of the diffeological tensor product $V\otimes W$ is the same as that of the usual tensor product, it is isomorphic to $V$. Therefore there
is \textbf{not} an isomorphism between $V\otimes W$ and $L^{\infty}(V^*,W)$, the two spaces being different as sets. On the other hand, $L^{\infty}(W^*,V)$ and
$V\otimes W$ are isomorphic as usual vector spaces; it is easy to see that they are also diffeomorphic (this follows from the fact that $V$ has the coarse
diffeology\footnote{Consider the obvious map $F:V\to L(\matR,V)=L^{\infty}(\matR,V)$ given by $F(v)(x)=xv$; it is obviously bijective, and it is smooth by
Proposition \ref{criterio-funct-diff}. Indeed, for any plot $p:U\to V$ we need that $F\circ p$ be a plot, which is equivalent to the map $U\times\matR\to V$
given by $(u,x)\mapsto(F\circ p)(u)(x)=xp(u)$ being smooth. But simply due to the fact that it is a map in $V$, that has the coarse diffeology, it is a plot of
it, so the conclusion.}).
\end{example}

The example just made shows that in general, at least one of these classical isomorphisms might fail to exist (and at a very basic level). We may wish however to
see what could be kept of the standard isomorphisms, in the sense that the two maps $V\otimes W\to L(V^*,W)$ and $V\otimes W\to L(W^*,V)$ are still defined; we
wonder if their ranges consist of smooth maps and, if so, whether they are smooth.

\begin{prop}
Let $V$, $W$ be two finite-dimensional diffeological vector spaces. Then:
\begin{enumerate}
\item If $\hat{F}:V\otimes W\to L(V^*,W)$ is the map defined, via linearity, by $v\otimes w\mapsto[\hat{F}(v\otimes w)(f)=f(v)w]$ then $\hat{F}$ takes values in
$L^{\infty}(V^*,W)$. Furthermore, as a map $V\otimes W\to L^{\infty}(V^*,W)$ between diffeological spaces, it is smooth;
\item If $\hat{G}:V\otimes W\to L(W^*,V)$ is the map defined, via linearity, by $v\otimes w\mapsto[\hat{G}(v\otimes w)(g)=g(w)v]$ then $\hat{G}$ takes values in
$L^{\infty}(W^*,V)$. Furthermore, as a map $V\otimes W\to L^{\infty}(W^*,V)$ between diffeological spaces, it is smooth.
\end{enumerate}
\end{prop}

\begin{proof}
Let us prove 1. We need to show that $\hat{F}$ is a smooth map that takes values in $L^{\infty}(V^*,W)$. To prove the latter, it is enough to show that
$\hat{F}(v\otimes w)$ is smooth, for any $v\in V$ and $w\in W$. Let us fix $v\in V$ and $w\in W$; we need to show that for any plot $p:U\to V^*$ the composition
$\hat{F}(v\otimes w)\circ p$ is a plot of $W$. Writing explicitly $(\hat{F}(v\otimes w)\circ p)(u)=\hat{F}(v\otimes w)(p(u))=p(u)(v)w$, we recall that any
constant map on a domain is a plot for any diffeology, so the map $c_w:U\to W$ that sends everything in $w$ is a plot of $W$. Finally, the map
$(u,v)\mapsto p(u)(v)$ is a smooth map to $\matR$, by Proposition \ref{criterio-funct-diff} and because $p$ is a plot of $V^*=L^{\infty}(V,\matR)$ whose
diffeology is functional; recalling that multiplication by scalar is smooth for any diffeological vector space, we get the conclusion.

Let us now prove that $\hat{F}$ is a smooth map $V\otimes W\to L^{\infty}(V^*,W)$; by Proposition \ref{criterio-funct-diff} we need to prove that the induced map
$V^*\times U\to W$ is smooth.\footnote{Note the change in the order of factors, for formal purposes.} This map acts by sending each $(f,u)$ (where $f\in V^*$) to
$(\hat{F}\circ p)(u)(f)$ and so it writes as $(f,u)\mapsto(\mbox{\textsc{ev}}_{V^*}\otimes\mbox{Id}_W)(\mbox{Id}_{V^*}\times p)(f,u)$; the diffeology of $V^*$
being functional, the evaluation map is smooth, therefore so is $\hat{F}\circ p$, whence the conclusion.

The proof of 2 is completely analogous, so we omit it.
\end{proof}

\begin{oss}
Example \ref{tensor-not-functional-example} also illustrates that, in general, there is \textbf{not} an analogue of the classical isomorphism
$V^*\otimes V\cong L^{\infty}(V,V)$: it suffices to consider the same $V$, that is, $\matR^n$ with the coarse diffeology. Then the product on the left is the
trivial space, $V^*$ being the trivial space, whereas the space on the right consists of \emph{all} linear maps $V\to V$ (since the coarse diffeology includes
any map into $V$, all of these maps are automatically smooth).
\end{oss}

\paragraph{Tensor product of duals and the dual of a tensor product} Recall, once again, that for usual vector spaces there is a standard isomorphism
$V^*\otimes W^*\cong(V\otimes W)^*$; it turns out that this is also true for (finite-dimensional) diffeological vector spaces. Observe immediately that by
Theorem \ref{dual-is-standard-thm} and the definition of the tensor product diffeology, both $V^*\otimes W^*$ and $(V\otimes W)^*$ are standard spaces, so for
them to be diffeomorphic it is sufficient that they have the same dimension.

\begin{thm}\label{tensor-dual:dual-tensor:thm}
Let $V$, $W$ be two finite-dimensional diffeological vector spaces. Then $\dim((V\otimes W)^*)=\dim(V^*)\cdot\dim(W^*)$.
\end{thm}

We will give a proof of this statement later in Section 2.

\subsection{Smooth direct sum decompositions}

In general, if we have a finite-dimensional diffeological vector space $V$ that, as a usual vector space, decomposes into a direct sum $V=V_1\oplus V_2$, the
following two diffeologies are \emph{a priori} different: the given diffeology of $V$, and the vector space diffeology on $V_1\oplus V_2$ resulting from the
subset diffeologies on $V_1$ and $V_2$. A simple example (Example 3.5 of \cite{me2018}) can be used to illustrate that the two diffeologies are indeed generally
different.

\subsubsection{Smoothness of direct sums}

To illustrate the issue, let us first cite the aforementioned example.

\begin{example} \emph{(\cite{me2018})}
Let $V$ be $\mathbb{R}^3$ endowed with the vector space diffeology generated by the $p:\matR\to V$ defined by $p(x)=|x|(e_2+e_3)$, where $\{e_1,e_2,e_3\}$ is
the canonical basis of $\matR^3$. Let $V_1=\mbox{Span}(e_1,e_2)$ and $V_2=\mbox{Span}(e_3)$. The subset diffeology on any of these two subspaces is then the
standard diffeology of $\matR^2$ and $\matR$ respectively, therefore their direct sum is again a standard space, while $V$ is not so.
\end{example}

It thus makes sense to have a separate term for cases when the two diffeologies that \emph{a priori} are different, do coincide.

\begin{defn}
Let $V$ be a diffeological vector space, and let $V_1,\,V_2$ be two its subspaces such that $V$ is isomorphic as a vector space to the direct sum $V_1\oplus V_2$.
We say that the the decomposition $V=V_1\oplus V_2$ is \textbf{smooth} if the diffeology of $V$ coincides with the direct sum diffeology on $V_1\oplus V_2$
relative to the subset diffeologies on $V_1$ and $V_2$ (\emph{i.e.}, if at least one isomorphism is also a diffeomorphism). We also say that a (diffeological,
\emph{i.e.} considered with the subset diffeology) subspace $W\leqslant V$ \textbf{splits off as a smooth direct summand} if there exists another subspace
$Z\leqslant V$ which is a direct complement of $W$ in $V$, and the decomposition $V=W\oplus Z$ is smooth.
\end{defn}

While in a usual finite-dimensional vector space all subspaces admit direct complements, a diffeological vector space may contain subspaces that do not split
off as smooth summands (which justifies introducing a separate term for them).

\begin{example}\label{not-a-smooth-summand-ex}
Let $V$ be as in the previous example, \emph{i.e.}, $\matR^3$ with the vector space diffeology generated by the map $p:\matR\to V$ acting by $p(x)=|x|(e_2+e_3)$.
By definition of this diffeology, the local shapes of its plots are precisely maps of form,
$$U\ni u\mapsto\left(\alpha_1(u),\alpha_2(u)+\sum_{i=1}^kf_i(u)|F_i(u)|,\alpha_3(u)+\sum_{i=1}^kf_i(u)|F_i(u)|\right),$$
where $\alpha_1,\alpha_2,\alpha_3,f_i,F_i\in C^{\infty}(U,\matR)$ are ordinary smooth maps.

Let $W=\mbox{Span}(e_2)$; it is a standard space for the subset diffeology. Suppose it splits off as a smooth direct summand, and let $W_1$ be its smooth direct
complement. Then all plots of $V$ locally have form $U\ni u\mapsto (0,\alpha(u),0)+Q(u)$, where $\alpha\in C^{\infty}(U,\matR)$ and $Q:U\to W_1$ is a plot for
a subset diffeology of $W_1$. This in particular implies that any map of form
$$U\ni u\mapsto\left(\alpha_1(u),\sum_{i=1}^kf_i(u)|F_i(u)|,\alpha_3(u)+\sum_{i=1}^kf_i(u)|F_i(u)|\right)$$
for any choice of ordinary smooth maps $\alpha_1,\alpha_3,f_i,F_i$.

Let $k=1$, and let $U$ be any domain. Let first $F_1$ be the constant map with the zero value, $F_1\equiv 0$. Choosing first $\alpha_1\equiv 1$ and
$\alpha_3\equiv 0$, we obtain that $W_1$ contains the vector $(1,0,0)$; choosing $\alpha_1\equiv 0$ and $\alpha_3\equiv 1$, we conclude that $W_1$ also contains
$(0,0,1)$. Let now $F_1\equiv 1$, and assume the same for $f_1$, that is, $f_1\equiv 1$. Let $\alpha_1\equiv\alpha_3\equiv 0$. This means that $W_1$ contains
the vector $(0,1,1)$, and since it also contains $(0,0,1)$, it contains $(0,1,0)$ as well. Therefore $W_1$ coincides with the whole $V$, which is a
contradiction.
\end{example}

The simplicity of the above example, and the fact that the proof involved used only constant maps and (implicitly) the mere existence of a nonstandard plot,
but almost not at all its specific shape, suggests that such examples may be abundant among diffeological vector spaces that are not standard. This in turn
suggests the possibility that there may exist diffeological vector spaces that admit no non-trivial smooth decompositions at all. Although we do not have an
example of such at hand, in view of this \emph{a priori} possibility we introduce a specific term for such spaces, calling them \textbf{unsplittable} (o
\textbf{irreducible}) spaces. Such spaces have at least one common property in that they do not admit non-trivial pseudo-metrics, the diffeological counterparts
of scalar products (see the next section). Notice that this is not a characterizing property meaning that a nonsplittable space certainly does not admit
a pseudo-metric, but a space that does not admit one, may have smooth decompositions.

\begin{example}\label{splittable-space-without-pseudometric-ex}
Let $V$ be $\matR^2$ endowed with the vector space diffeology generated by the plots $p_i:\matR\to V$ for $i=1,2$, acting by $p_i(x)=|x|e_i$. It is quite easy
to see (in fact, obvious from construction) that the decomposition of $V$ into the direct sum of $\mbox{Span}(e_1)$ and $\mbox{Span}(e_2)$ is a smooth one. As 
easily follows from the results in Section 2, the only pseudo-metric on this space is the trivial one.
\end{example}

The overall conclusion is that in a finite.dimensional diffeological vector space splitting off of a given subspace as a smooth direct summand is not automatic
and usually needs to be imposed as an assumption. Some partial criteria do exist.

\begin{lemma} \emph{(Lemma 3.6, \cite{me2018})}
Let $V$ be a finite-dimensional diffeological vector space, let $V_1\leqslant V$ be a subspace such that its subset diffeology is standard, and let
$V=V_1\oplus V_2$ be its vector space decomposition into a direct sum for some subspace $V_2\leqslant V$. This decomposition is smooth if and only if there
exists a basis $\{v_1,\ldots,v_k\}$ of $V_1$ such that the projection on each vector $v_i$ is a smooth linear functional on $V$.
\end{lemma}

\subsubsection{Smooth direct sums and diffeological dual}

For usual vector spaces the dual of a subspace is canonically contained in the dual of the entire space, via the extension by zero to any direct complement.
For diffeological vector spaces this \emph{a priori} may not be true, since a function smooth on a subspace may fail to extend smoothly to the rest of the
space.

\begin{example}
Let $V$ and $W$ be as in Example \ref{not-a-smooth-summand-ex}. The obvious generator of the dual of $W$ extends, in the way just mentioned, to the function
$e^2$ defined on the whole $V$, which however is not smooth, since its composition with the plot $p$ is the absolute value function, $(e^2\circ p)(x)=|x|$.
\end{example}

Such a problem does not present itself for subspaces that split off smoothly.

\begin{lemma}\label{embedding-a-dual-lem}
Let $V$ be a finite-dimensional diffeological vector space, and let $W\leqslant V$ split off as a smooth direct summand. Then $W^*$ canonically embeds into
$V^*$.
\end{lemma}

\begin{proof}
Let $W'\leqslant V$ be a smooth direct complement of $W$, and let $w^*\in W^*$. It suffices to show that the linear function $\hat{w}^*$ defined by requiring it
to coincide with $w^*$ on $W$ and vanish on $W'$ is smooth. Indeed, let $p:U\to V$ be a plot of $V$; then locally it has form $p=q+q'$, where $q$ is a plot of
$W$ and $q'$ is one of $W'$. Then by construction $\hat{w}^*\circ p=w^*\circ q$, which is smooth by choice of $w^*$.
\end{proof}

We stress that due to this canonical embedding, $W^*$ can be identified with any subset of $V^*$ consisting of all the functions that vanish on some fixed
smooth direct complement of $W$ in $V$. Observe however that \emph{a priori} such identification may not be unique (so word canonically refers to the
procedure of constructing such embedding, not to the embedding itself), since $W$ may have more than one smooth complement; for instance, in the
Example \ref{more-than-one-splitting-ex} (see below) the subspace $\mbox{Span}(e_1)$ has both $\mbox{Span}(e_2)$ and $\mbox{Span}(e_1+e_2)$ as smooth direct
complements. The resulting embedding of $W^*$ into $V^*$ may therefore depend on the specific choice of the smooth direct complement of $W$ (although this
does not occur in the just-cited example, since the dual of $\mbox{Span}(e_1)$ is trivial).

\begin{cor}
Let $V$ be a finite-dimensional diffeological vector space, and let $V=V_1\oplus V_2$ be a smooth direct sum decomposition of $V$. Then
$V^*\cong V_1^*\times V_2^*$, and the diffeomorphism is obtained by assigning to each smooth linear function on $V$ its restrictions onto $V_1$ and $V_2$.
\end{cor}

\begin{proof}
The main ingredient in the proof is showing that the map thus defined is surjective, and this is ensured by Lemma \ref{embedding-a-dual-lem}.
\end{proof}

If we consider $V_1^*$ and $V_2^*$ as subspaces of $V^*$ via the above lemma, we can also say that $V^*$ splits as the direct sum of $V_1^*$ and $V_2^*$.

\subsubsection{Smooth direct sums and linear maps}

For usual finite-dimensional vector spaces there is a type of their standard representations as direct sums: any linear map defined on a given vector space
allows to represent it as the direct sum of the map's kernel and the map's image, although there are infinitely many ways to do that. On the other hand, for
diffeological spaces it should already be quite clear that no natural analogue of this statement would be true in general. Indeed, the most natural analogue
would be, that the domain space should be diffeomorphic to the direct sum of the map's kernel and the map's image considered with the subset diffeology relative
to the diffeology on the target space. However, we have just seen (Example \ref{not-a-smooth-summand-ex}) that some subspaces do not split off as smooth
direct summands, while for any given subspace it is quite easy to construct a smooth linear map whose kernel is precisely that subspace.

\begin{example}
Let $V$ be the same as in Example \ref{not-a-smooth-summand-ex}, and let $W$ be $\matR^2$ endowed with the diffeology that is the pushforward of the diffeology
of $V$ by the linear map $f:V\to W$ such that $f(e_1)=e_1$, $f(e_2)=0$, $f(e_3)=e_2$ (here $e_1,e_2,e_3$ on the left form the canonical basis of $V$, and
$e_1,e_2$ on the right form the canonical basis of $W$). The map $f$ is smooth and linear by construction, and its kernel is $\mbox{Span}(e_2)$ --- a subspace
that, as we have previously shown, does not split off as a smooth direct summand.
\end{example}

Thus, in order to have some hope to obtain a diffeomorphism with $\mbox{Ker}(f)\oplus\mbox{Im}(f)$ for a given $f:V\to W$ (where $V,W$ are diffeological vector
spaces and $f$ is a smooth linear map) we must at least ensure that $\mbox{Ker}(f)$ splits off as a smooth direct summand. However, in general this is not
sufficient, due to the issue with the diffeology on $\mbox{Im}(f)$. The issue basically is that the existence of a diffeomorphism between the domain space
and $\mbox{Ker}(f)\oplus\mbox{Im}(f)$ would at least imply that the natural diffeology on $\mbox{Im}(f)$ (the subset diffeology relative to that on the target
space) must coincide with the pushforward of the diffeology on the domain space. However, the subset diffeology on $\mbox{Im}(f)$ might easily be larger; for a
simple example, it suffices for $V$ to be a standard space of dimension at least $2$, for $W$ to be a coarse space of dimension at least $1$, and for $f$ to any
non-zero linear map with a non-trivial kernel.

The issue of the diffeology of the target space being arbitrarily large must be pretty much excluded by assumption, just as the possibility of the kernel not
splitting off as a smooth direct summand needs to be. However, with assumptions of this kind in place, the two issues are to some extent interrelated. In
particular, the following is true.

\begin{lemma} \emph{(Lemma 4.5 of \cite{me2019})}
Let $V$ be a finite-dimensional diffeological vector space, let $W$ be a vector space, and let $f:V\to W$ be a surjective linear map. Let $V'\leqslant V$ be a
subspace of $V$, and let $W'=f(V')$ be the corresponding subspace of $W$. Endow $W$ with the pushforward of the diffeology of $V$ by the map $f$; then the
corresponding subset diffeology on $W'$ is in general coarser than the pushforward of the subset diffeology on $V'$ by the map $f|_{V'}$. If $\mbox{Ker}(f)$
splits off as a smooth direct summand then the two diffeologies coincide.
\end{lemma}

Note that the two diffeologies on an image may happen to coincide also without the kernel splitting off as a smooth summand, so the relation is one-way only.
However, in view of all of the above, this appears the best that one can obtain.

Finally, observe that under the appropriate assumptions smooth direct sums are well-behaved with respect to smooth linear maps. More precisely, the following
is true,

\begin{lemma} \emph{(Lemma 4.7 of \cite{me2019})}
Let $V$ be a finite-dimensional diffeological vector space, let $W$ be a vector space, and let $f:V\to W$ be a surjective linear map; endow $W$ with the
pushforward diffeology. Let $V=V_0\oplus V_1$ be a smooth decomposition of $V$ into a direct sum, and suppose that $f(V_0)\cap f(V_1)=\{0\}$. Then the
decomposition $W=W_0\oplus W_1$, where $W_i=f(V_i)$ for $i=0,1$, is smooth.
\end{lemma}

The assumption that the diffeology of $W$ is the pushforward diffeology, is essential; making it larger would keep $f$ smooth (which it would not be for a
smaller one) but might prevent the conclusion from being true.

\section{Diffeological pseudo-metrics on vector spaces}

For a finite-dimensional diffeological vector space, the analogue of a scalar product is (probably) the most obvious one: it is a minimally degenerate smooth
bilinear form. This choice of the corresponding notion of a metric is pretty much a forced one.

\subsection{What is a pseudo-metric}

A pseudo-metric is the closest analogue of a smooth scalar product that can exist on finite-dimensional diffeological vector space.

\begin{defn}
Let $V$ be a diffeological vector space of finite dimension $n$. A \textbf{diffeological pseudo-metric} on $V$ is a semi-positive definite smooth symmetric
bilinear form on $V$ such that its maximal isotropic subspace has dimension $n-\dim(V^*)$, where $V^*$ is the diffeological dual of $V$.
\end{defn}

If $V$ has a fixed identification with $\matR^n$, any symmetric bilinear form is represented by a symmetric $n\times n$ matrix $A$. Assuming the form is smooth,
the matrix $A$ defines a pseudo-metric if and only if the multiplicity of its zero eigenvalue is equal to $n-\dim(V^*)$, and all the other eigenvalues are
positive. Notice that, as was shown in Lemma 3.2 of \cite{me2018}, for an arbitrary smooth symmetric bilinear form on $V$ the multiplicity of the eigenvalue $0$
is always at least $n-\dim(V^*)$, so a pseudo-metric is a form that achieves the minimal value of the multiplicity.

At least one pseudo-metric always exists on any finite-dimensional diffeological space $V$ (in the most degenerate case, \emph{i.e.}, when the dual space is
trivial, there is precisely one pseudo-metric, which is trivial). Indeed, let $f^1,\ldots,f^k$ be any basis of $V^*$; then $\sum_{i=1}^kf^i\otimes f^i$ is
obviously a pseudo-metric on $V$.

\begin{example}\label{all-pseudo-metrics-ex}
Let $V$ be $\matR^2$ endowed with the vector space diffeology generated by the plot $p:\matR\to V$ acting by $p(x)=|x|e_1$. Let us show that any pseudo-metric
on $V$ is given by a positive multiple of the matrix $\left(\begin{array}{cc} 0 & 0 \\ 0 & 1 \end{array}\right)$.

Consider a generic symmetric $2\times 2$ matrix $A$, which has form $\left(\begin{array}{cc} a & b \\ b & c \end{array}\right)$; assume that not all of $a,b,c$ are
zero. For this matrix to define a pseudo-metric on $V$, it is necessary and sufficient that the evaluation of the corresponding bilinear form on any given pair
of plots be an ordinary smooth function to $\matR$. Evaluting it on $p$ and the constant plot with value $e_1$ yields the function $x\mapsto a|x|$, which
implies that $a=0$. Likewise, the evaluation on $p$ and the constant plot with value $e_2$ allows to conclude that $b=0$. Finally, evaluating it on a generic
pair of plots implies that $c$ can be any positive number. Indeed, a generic plot $Q$ of $V$ has form
$$U\ni u\mapsto(\alpha_1(u)+\sum_if_i(u)|F_i(u)|,\alpha_2(u)),$$
so evaluating the bilinear form associated to $A$ on a pair of such plots yields a function of form $u\mapsto c\alpha_2^{(1)}(u)\alpha_2^{(2)}(u)$, where
$\alpha_2^{(1)},\alpha_2^{(2)}$ are the compositions of the two plots with the projection onto the second coordinate. Such function is necessarily smooth
for any $c$. Finally, evaluating the form on the pair both of whose components are constant plots with value $e_2$, implies $c$ can, and must, be any positive
number.
\end{example}

\subsection{Pseudo-metrics and the diffeological dual}

Similarly to the standard vector spaces and scalar products on them, a pseudo-metric on a diffeological vector space induces a natural surjective map onto the
diffeological dual, which turns out to be smooth, via the assignment to each $v\in V$ of the smooth linear functional $g(\cdot,v)$. Except for the standard case,
it is not a diffeomorphim; however, it does admit left inverses.

\begin{thm}\label{map-onto-dual-thm}
Let $V$ be a finite-dimensional diffeological vector space, and let $g$ be a pseudo-metric on $V$. The induced map $\Psi:V\to V^*$ (\emph{i.e.}, the map given
by $v\mapsto[w\mapsto g(w,v)]$) is a smooth linear surjective map which admits a smooth left inverse. In particular, the restriction of $\Psi$ onto the image
of this inverse is a diffeomorphism onto $V^*$.
\end{thm}

\begin{proof}
Choose a basis $u_1,\ldots,u_n$ of $V$, and consider the (usual vector space) isomorphism $\varphi:V\to\matR^n$ corresponding to this basis. Endow this copy of
$\matR^n$ with the diffeology that is the pushforward of the diffeology of $V$ by $\varphi$. Let $V_{coor}$ be the diffeological vector space thus obtained. By
construction, $\varphi$ is then a diffeomorphism $V\to V_{coor}$. In particular, $g$ induces via $\varphi$ a pseudo-metric on $V_{coor}$ given by a symmetric
square $n\times n$ matric $A$.

Let $v_1,\ldots,v_n$ be an orthonormal basis of eigenvectors of $A$, in which the first $k$ vectors belong to positive eigenvalues and the rest belong to the
eigenvalue $0$. Let $\lambda_1,\ldots,\lambda_k$, taken with multiplicity, denote all the positive eigenvalues of $A$, \emph{i.e.} $Av_i=\lambda_iv_i$ for
$i=1,\ldots,k$, and let $V_0$ be the subspace generated by $v_1,\ldots,v_k$.

We claim that the restriction of $\Psi$ onto $\varphi^{-1}(V_0)$ is a diffeomorphism. To prove this, it is of course sufficient to prove the analogous statement
for $V_0\leqslant V_{coor}$ and the map $\Psi_A$ induced by the natural pairing determined by the matrix $A$. This was actually proved in Theorem 3.11 of
\cite{me2018}, but since that proof used an erroneous statement (easily avoidable and of which we provide details later in this paper), we reprove it here.

Observe first of all that the subset diffeology of $V_0$ is standard. Indeed, let $p:U\to V$ be a plot of $V$ such that its image is contained in $V_0$. Let us
write $p(u)=p_1(u)v_1+\ldots+p_k(u)v_k$, where $p_i$ is obviously the composition of $p$ with the projection (in the usual sense) on $v_i$. Recall now that the
map $u\mapsto\langle p(u)|v_i\rangle_A$ must be smooth in the ordinary sense by definition of a pseudo-metric. However, for all $i=1,\ldots,k$ this map is
actually $u\mapsto\lambda_ip_i(u)$ with $\lambda_i\neq 0$ by choice of $v_1,\ldots,v_k$. Hence all maps $p_i$ are ordinary smooth maps, and therefore so is $p$.
This means precisely that the subset diffeology of $V_0$,

We claim furthermore that $V_0$ splits off as a smooth direct summand, and for any such splitting the diffeological dual of the other summand is trivial. We
first show that the decomposition of $V_{coor}$ into the direct sum of $V_0$ and $V_1$, the subspace generated by the eigenvectors of $A$ belonging to the
eigenvalue $0$, is smooth.

Let $p:U\to V_{coor}$ be any plot of $V_{coor}$; write $p(u)=p_1(u)v_1+\ldots+p_n(u)v_n$, where $p_i:U\to\matR$ is the composition of $p$ with the projection
onto $v_i$ (that is, onto $\mbox{Span}(v_i)$). Notice that $p$ is thus the sum $p=p_0+q$, where $p_0$ and $q$ are given by $p_0(u)=p_1(u)v_1+\ldots+p_n(u)v_k$
and $q(u)=p_{k+1}(u)v_{k+1}+\ldots+p_n(u)v_n$ and take values in $V_0$ and $V_1$ respectively. It thus suffices to show that $p_0$ is a plot of $V_0$ to be
sure that $q=p-p_0$ is one of $V_1$ and so to obtain the desired conclusion. But since $V_0$ has standard diffeology, $p_0$ being a plot of it is equivalent
to all $p_1,\ldots,p_k$ being ordinary smooth functions. To conclude, observe as before that $u\mapsto\langle p(u)|v_i\rangle_A$ is ordinary smooth by
construction of $A$ and that we still have $\langle p(u)|v_i\rangle_A=\lambda_ip_i(u)$, so the decomposition $V_{coor}=V_0\oplus V_1$ is indeed smooth. (A similar
reasoning appears in Theorem 3.6 of \cite{wu}, its corollary Lemma 3.6, as well as Lemma 3.7, of \cite{me2018}).

Let now $V_{coor}=V_0\oplus W_1$ be any smooth decomposition of $V_{coor}$; let us show that $W_1^*$ is trivial. Assume that it is not, and let $w^*\in W_1^*$
be non-trivial. Since $V_0$ has standard diffeology, its diffeological dual coincides with the usual dual and has dimension equal to $k=\dim(V_{coor}^*)$. Choose
a basis $v_1^*,\ldots,v_k^*$ of $V_0^*$ and consider the corresponding elements $v^1,\ldots,v^k$ of $V_{coor}^*$ obtained by defining them to act as
$v_1^*,\ldots,v_k^*$ respectively on $V_0$ and to vanish on $W_1$; let also $w^1$ act as $w^*$ on $W_1$ and vanish on $V_0$. All of the functions
$v^1,\ldots,v^k,w^1$ are smooth because the decomposition $V_{coor}=V_0\oplus W_1$ is smooth, and are linear and linearly independent by construction. Therefore
the symmetric bilinear form $v^1\otimes v^1+\ldots+v^k\otimes v^k+w^1\otimes w^1$ is smooth and semipositive definite, and has rank $k+1>\dim(V_{coor}^*)$, which
contradicts Lemma 3.2 of \cite{me2018}.

Notice finally that $V_0$ is a maximal subspace such that its subset diffeology is standard and it splits off as a smooth direct summand. Observe more in
general that if $V_{coor}=V_1\oplus W_1$ is a smooth decomposition such that the subset diffeology of $V_1$ is standard, there exists a smooth symmetric bilinear
form on $V_{coor}$ of rank equal to $\dim(V_1)$. Indeed, to define such form it suffices to pose it to be a scalar product on $V_1$, to vanish on $W_1$, and
require $V_1$ and $W_1$ to be orthogonal. So if now we assume that $V_0$ is contained in a larger subspace $V_1'$ that has standard diffeology and splits off as
a smooth direct summand, we can construct a smooth symmetric bilinear form of rank strictly greater than $\dim(V_0)=\dim(V_{coor}^*)$, which is again a
contradiction. (This proof was given in Proposition 3.8 of \cite{me2018}).

It is easy to conclude from the above paragraph that the diffeological dual of $V_{coor}$ is diffeomorphic to $V_0$. Indeed, let $V_{coor}=V_0\oplus V_1$ be a
smooth direct sum decomposition for some subspace $V_1$. Then
$$V_{coor}^*=(V_0\oplus V_1)^*\cong V_0^*\times V_1^*\cong V_0^*\cong V_0,$$
where the last diffeomorphism is due to $V_0$ being standard.

It remains to show that such a diffeomorphism can be given by the restriction of $\Psi_A$ onto $V_0$. By construction, this restriction is an isomorphism;
observe in particular that by definition of a pseudo-metric $V_0$ and $V^*$ have the same dimension. It is thus sufficient to prove that it is both ways
smooth.

Let $v^i=\Psi_A(v_i)$ for $i=1,\ldots,k$, where $v_1,\ldots,v_k$ is the already used orthonormal basis of $V_0$ consisting of certain eigenvectors of $A$, and
let us write $\Psi_A^{-1}$ for the linear map $V^*\to V_0$ given by $\Psi_A^{-1}(v^i)=v_i$ for $i=1,\ldots,k$ (that is, $\Psi_A^{-1}=(\Psi_A|_{V_0})^{-1}$).
Recall that a generic plot $p$ of $V_0$ locally has form $p(u)=p_1(u)v_1+\ldots+p_k(u)v_k$ for some ordinary smooth functions $p_1,\ldots,p_k$. Since
$\Psi_A\circ p$ acts by $(\Psi_A\circ p)(u)=p_1(u)v^1+\ldots+p_k(u)v^k$, it is a plot of $V^*$ (since $p_1,\ldots,p_k$ are smooth, this kind of statement would
be true for any vector space with a vector space diffeology). And since $V^*$ also has standard diffeology, its generic plot $q$ also has form
$u\mapsto q_1(u)v^1+\ldots+q_k(u)v^k$ for some smooth functions $q_1,\ldots,q_k$, with the composition $\Psi_A^{-1}\circ q$ acting by
$(\Psi_A^{-1}\circ q)(u)=q_1(u)v_1+\ldots+q_k(u)v_k$, so a plot of $V_0$, and the theorem is proven.
\end{proof}

The construction in the above proof depends \emph{a priori} on the choice of the auxiliary basis $u_1,\ldots,u_n$; a different choice of the basis may produce
a different inverse.

\begin{example}\label{characteristic-not-unique-ex}
Consider $V$ as in the Example \ref{all-pseudo-metrics-ex}, and consider the pseudo-metric $g$ on it given by the matrix
$\left(\begin{array}{cc} 0 & 0 \\ 0 & 1 \end{array}\right)$. Choose $u_1=e_1$, $u_2=e_1+e_2$ as a basis (we will show in a subsequent example that it is actually
a smooth basis, \emph{i.e.} the decomposition of $V$ into the direct sum of the spans of its components is smooth). The matrix of $g$ written with respect to
this basis is again $\left(\begin{array}{cc} 0 & 0 \\ 0 & 1 \end{array}\right)$, but now the unitary eigenvector $\left(\begin{array}{c} 0\\ 1\end{array}\right)$
belonging to its positive eigenvalue consists of the coordinates of a generator $v$ of $V_0$, the target space of the left inverse of the induced map $\Psi$, in
the basis $u_1,u_2$; therefore this generator is actually $v=e_1+e_2$, so $V_0=\mbox{Span}(e_1+e_2)$. But if we were to use the canonical basis of $\matR^2$ to
determine it, we would have obtained $V_0=\mbox{Span}(e_2)$.
\end{example}

The choice of the left inverse of the map induced by natural pairing depends therefore on the choice of the basis of $V$, that is, on its identification with
some Euclidean space.

\begin{cor}\label{unique-inverse-of-psi-cor}
Let $V$ be a diffeological vector space whose underlying vector space is $\matR^n$ for some $n$. Then there is a canonical choice $\Psi_0^{-1}$ of a left inverse
of $\Psi$.
\end{cor}

\begin{proof}
A pseudo-metric on $\matR^n$ (considered as a diffeological space) is given by a well-defined matrix $A$, whose set of eigenvectors is uniquely defined. This
implies that the subspace $V_0$ generated by the eigenvectors belonging to the positive eigenvalues is uniquely defined, and therefore so is the left inverse
of the induced map $\Psi$ constructed above.
\end{proof}

Finally, the map $\Psi$ induced by the natural pairing allows to carry over to $V^*$ the initial pseudo-metric, which becomes a true scalar product on $V^*$
(see Corollary 3.13 of \cite{me2018}). Furthermore, $\Psi$ allows to easily obtain the following result.

\begin{cor} \label{max-dim-of-standard-subspace-cor}
Let $V$ be a finite-dimensional diffeological vector space, and let $V'$ be any its subspace whose subset diffeology is standard. Then
$\dim(V')\leqslant\dim(V^*)$.
\end{cor}

\begin{proof}
Endow $V$ with any pseudo-metric and consider the corresponding map $\Psi$ induced by natural pairing. It suffices to observe that the restriction of $\Psi$
to $V'$ is injective.
\end{proof}

\subsection{The maximal isotropic subspace relative to a pseudo-metric}

As a generally degenerate symmetric bilinear form, a pseudo-metric determines its maximal isotropic subspace.

\begin{lemma}\label{maximal-isotropic-subspace-lem}
Let $V$ be a finite-dimensional diffeological vector space, let $g$ be a pseudo-metric on $V$, and let $V_1$ be the maximal isotropic subspace of $g$. Then
$V_1$ splits off smoothly.
\end{lemma}

\begin{proof}
Let $\Psi:V\to V^*$ be the natural pairing map induced by $g$ (\emph{i.e.}, acting by $v\mapsto[w\mapsto g(w,v)]$), and let $\hat{\Psi}_0^{-1}$ be any choice of
its left inverse. Denote $V_0:=\hat{\Psi}_0^{-1}(V^*)$, the pre-image of $V^*$. We claim that $V$ splits as a smooth direct sum of $V_0$ and $V_1$.

Since $V_1$ is precisely the kernel of the linear map $\Psi$, and $V_0$ is diffeomorphic to $\mbox{Im}(\Psi)$ via its left inverse $\hat{\Psi}_0^{-1}$, the vector
space underlying $V$ obviously coincides the direct sum of its vector space subspaces $V_0$ and $V_1$. What we really need to show is that the corresponding
direct sum decomposition $V_0\oplus V_1$ is smooth; for that, it suffices to show that every plot $p$ of $V$ is locally the sum of some plot $p_0$ of $V_0$ and
some plot $q$ of $V_1$.

Define $p_0:=\hat{\Psi}_0^{-1}\circ\Psi\circ p$; since $\hat{\Psi}_0^{-1}$ and $\Psi$ are smooth and $p$ is a plot, $p_0$ is again a plot of $V$. Furthermore,
since its image is contained in $V_0$ by construction, it is a plot for the subspace diffeology on $V_0$. It remains to define $q$ to be $q=p-p_0$; this is a
plot of $V$, since the operations are smooth in a vector space diffeology, and its image is obviously contained in $V_1$, whence the claim.
\end{proof}

As follows from the discussion in the previous section, the decomposition of $V$ thus obtained is \emph{a priori} not unique. However, if the underlying vector
space of $V$ is $\matR^n$, by taking $\hat{\Psi}_0^{-1}$ to be $\Psi_0^{-1}$, the canonically defined left inverse of Corollary \ref{unique-inverse-of-psi-cor},
we obtain a well-defined canonical decomposition.

\subsection{Characteristic subspaces}

Let $V$ be a finite-dimensional diffeological vector space endowed with a pseudo-metric $g$ such that its underlying vector space is $\matR^n$ for some $n$.
We define its \textbf{characteristic subspace} to be the pre-image $\Psi_0^{-1}(V^*)$ of the diffeological dual under the map $V\to V^*$ induced by the natural
pairing. By definition, the dimension of the characteristic subspace is equal to that of the diffeological dual $V^*$ of $V$. Obviously, for a standard space any
its characteristic subspace is the space itself. On the other extreme, it may also happen that any characteristic subspace is trivial; this is the case of
Example \ref{splittable-space-without-pseudometric-ex}. Furthermore, we have the following.

\begin{prop} \emph{(Proposition 3.8, Theorem 3.12 of \cite{me2018})}
Let $V$ be a finite-dimensional diffeological vector space endowed with a pseudo-metric $g$. The characteristic subspace of $V$ is a maximal subspace of $V$ that
is standard and splits off as a smooth direct summand.
\end{prop}

Recall that, if $A$ is the matrix that defines the pseudo-metric $g$, the characteristic subspace is the subspace generated by all the eigenvectors of $A$
belonging to its positive eigenvalues. It would be nice if the characteristic subspace of $V$ could also be independently determined by the property of being a
maximal standard subspace that splits off smoothly, but (contrary to what we originally thought\footnote{The proof of Proposition 3.10 of \cite{me2018} contains
a mistaken claim in asserting that for $w_0=v_0+v_1$ with $w_0$ and $v_0$ such that the projections on them are smooth linear functionals, so is the projection
on $v_1$. There is of course no reason to claim this, and the specific instance of where this is not true, is given by the Example
\ref{more-than-one-splitting-ex}, for $w_0=e_1+e_2$, $v_0=e_2$, and $v_1=e_1$.}) this is not the case.

\begin{example}\label{more-than-one-splitting-ex}
Consider again $V$ of Example \ref{all-pseudo-metrics-ex}, \emph{i.e.} $\matR^2$ endowed with the vector space diffeology generated by the plot $p:\matR\to V$
acting by $p(x)=|x|e_1$. We claim the following: 1) the diffeological dual of $V$ has dimension 1; 2) the subspace $V_0:=\mbox{Span}(e_2)$ is the characteristic
subspace of $V$ for the pseudo-metric $g$ given by the matrix $\left(\begin{array}{cc} 0 & 0 \\ 0 & 1 \end{array}\right)$; 3) the subset diffeology on the
subspace $V_1:=\mbox{Span}(e_1+e_2)$ is standard and splits off smoothly. Let us prove these claims.

To prove 1), it suffices to show that the diffeological dual $V^*$ is generated by the linear map $e^2$ (belonging to the canonical dual basis of $\matR^2$).
Indeed, recall that a generic plot $Q$ of $V$ has form
$$U\ni u\mapsto(\alpha_1(u)+\sum_if_i(u)|F_i(u)|,\alpha_2(u)),$$
where $U$ is a domain in some Euclidean space and $\alpha_1,\alpha_2,f_i,F_i\in C^{\infty}(U,\matR)$ are any. On the other hand, as a set, the diffeological dual
$V^*$ is, of course, a subset of the ordinary dual, so its elements have form $ae^1+be^2$; it suffices to show that $ae^1+be^2$ belongs to $V^*$ if and only if
$a=0$. By definition of the diffeological dual, $ae^1+be^2\in V^*$ is equivalent to $(ae^1+be^2)\circ Q$ being an ordinary smooth map $U\to\matR$. It suffices
to take $Q$ to be $p$, the generating plot; we have $(ae^1+be^2)\circ p$ is the map $x\mapsto a|x|$, which is smooth for $a=0$ only. On the other hand, for a
generic plot $Q$ of the above form, the composition $e^2\circ Q$ is just the function $\alpha_2$, smooth by assumption. Therefore $V^*$ coincides with
$\mbox{Span}(e^2)$ and in particular, it has dimension 1.

To prove 2), we need to show that the subset diffeology of $V_0$ is standard, and that it splits off as a smooth direct summand (it is then automatically maximal
for these properties, since its dimension is just one less than that of $V$, which itself, not being standard, cannot be its own characteristic subspace). Both
of these claims readily follow from the general form of a plot of $V$ shown above. Indeed, $Q$ is a plot of the subset diffeology of $V_0$ if and only if it has
form
$$u\mapsto(0,\alpha_2(u));$$
furthermore every $Q$ writes as a sum of plots
$$u\mapsto(\alpha_1(u)+\sum_if_i(u)|F_i(u)|,0)\,\,\mbox{ and }\,\,u\mapsto(0,\alpha_2(u)),$$
which is sufficient to show that the usual direct sum decomposition of $V$ into $\mbox{Span}(e_1)\oplus\mbox{Span}(e_2)$ is a smooth one.

Finally, let us show that the subset diffeology of $V_1$ is standard as well. It suffices to observe that a plot $Q$ of the form shown above belongs to the
subset diffeology of $V_1$ if and only if we have
$$\alpha_1(u)+\sum_if_i(u)|F_i(u)|=\alpha_2(u)\,\,\mbox{ for all }u\in U.$$
In particular, it is an ordinary smooth map.

The subspace $V_1$ is therefore a maximal standard subspace, since it has dimension just one less than that of the entire $V$, which is not standard. We now
observe that the direct sum decomposition $V=V_1\oplus\mbox{Span}(e_1)$ is smooth. This easily follows from the fact that a generic plot $Q$ of $V$ of the above
shape writes as the sum of plots
$$u\mapsto(\alpha_2(u),\alpha_2(u))\,\,\mbox{ and }\,\,u\mapsto(\alpha_1(u)-\alpha_2(u)+\sum_if_i(u)|F_i(u)|,0),$$
of which the former is a plot of $V_1$ and the latter is a plot of $\mbox{Span}(e_1)$.

The subspace $V_1$ is therefore also a maximal standard subspace that splits off smoothly. However, it is not the characteristic subspace of any pseudo-metric
on $V$. Indeed, we have shown in a previous example that any pseudo-metric on $V$ is a positive multiple of $g$, so all of them have the same characteristic
subspace, which is $V_0$.
\end{example}

\subsection{Proof of Theorem \ref{tensor-dual:dual-tensor:thm}}

Given two finite-dimensional diffeological vector spaces $V$ and $W$, fix their decompositions into smooth direct sum of a characteristic subspace and the
maximal isotropic subspace, $V=V_0\oplus V_1$ and $W=W_0\oplus W_1$. Choose bases $v_1,\ldots,v_k$ and $w_1,\ldots,w_l$ of $V_0$ and $W_0$ respectively, and
complete them to bases $v_1,\ldots,v_n$ and $w_1,\ldots,w_m$ of the entire $V$ and $W$.

Consider the corresponding basis $v_i\otimes w_j$ with $i=1,\ldots,n$ and $j=1,\ldots,m$ of $V\otimes W$. Let $v^i$ and $w^j$ be the usual dual bases of the
\emph{usual} duals of $V$ and $W$, \emph{i.e.}, $v^i(v_s)=\delta_{is}$ and $w^j(w_t)=\delta_{js}$; by the standard isomorphism, the collection of all
$v^i\otimes w^j$ forms the basis of the usual dual $(V\otimes W)^*$, of which the diffeological dual is a subset. Observe that by construction $v^1,\ldots,v^k$
is a basis of the diffeological dual $V^*$ and $w^1,\ldots,w^l$ is one of the diffeological dual $W^*$. It thus suffices to prove that $v^i\otimes w^j$ for
$i=1,\ldots,k$ and $j=1,\ldots,l$ form a basis of $(V\otimes W)^*$.

Let $f$ be a generic element of $(V\otimes W)^*$. Since it is also an element of the usual dual of $V\otimes W$, it writes (uniquely) as
$f=\sum_{i,j}a_{ij}v^i\otimes w^j$. It suffices to prove that $a_{ij}=0$ as soon as either $i>k$ or $j>l$ to obtain the claim of the theorem.

Suppose that $i>k$ and $j$ is any, and consider $\mbox{Span}(v_i)$. Since $v_i\in V_1$, the maximal isotropic subspace, the subset diffeology of
$\mbox{Span}(v_i)$ is non-standard. Since it is however a vector space diffeology, there exists at least one non-constant plot $p:U\to\mbox{Span}(v_i)$ of form
$p(u)=p_i(u)v_i$, where $p_i:U\to\matR$ is \emph{not} a smooth function. Let $q_j:U\to\{w_j\}$ be a constant plot of $W$ with value $w_j$, and $p\otimes q_j$
be the plot of $V\otimes W$ corresponding to the pair $p,q_j$ in the tensor product diffeology. Then the composition $f\circ(p\otimes q_j)$ coincides with
the function $a_{ij}p_i$, and since $f$ must be smooth, $a_{ij}$ must be zero. This is sufficient to establish the statement of the Theorem.\finedimo

\subsection{Independence of pseudo-metric}

We have shown above that every finite-dimensional diffeological vector space $V$ endowed with a pseudo-metric splits into a smooth direct sum of its maximal
isotropic subspace and its characteristic subspace. This splitting depends of course on the specific identification of the underlying vector space with $\matR^n$
for appropriate $n$ (so on the choice of the basis), since the characteristic subspace does (see Example \ref{characteristic-not-unique-ex}). However, we now show
that it does not depend on the choice of the pseudo-metric.

\begin{lemma}\label{maximal-isotropic-is-unique-lem}
The maximal isotropic subspace is an invariant with respect to the choice of a pseudo-metric on $V$.
\end{lemma}

\begin{proof}
We claim that the maximal isotropic subspace of any pseudo-metric is precisely the intersection of the kernels of all smooth linear functions on $V$ (\emph{i.e.},
it is the maximal subspace such that any smooth linear function vanishes on it). This follows from Theorem \ref{map-onto-dual-thm}. Obviously, such
characterization is wholly independent of the concept itself of a pseudo-metric.
\end{proof}

Let now the underlying vector space of $V$ be $\matR^n$.

\begin{thm}\label{characteristic-is-well-defined-thm}
The characteristic subspace of $V$ does not depend on the choice of a pseudo-metric.
\end{thm}

\begin{proof}
This follows immediately from Lemma \ref{maximal-isotropic-is-unique-lem}. Indeed, since eigenvectors belonging to different eigenvalues are always orthogonal
(with respect to the canonical scalar product on $\matR^n$), the characteristic subspace corresponding to any given pseudo-metric coincides with the usual
orthogonal complement of the maximal isotropic subspace. Since the latter is the same for all pseudo-metrics, so is the former.
\end{proof}

Thus, any identification of the underlying vector space $|V|$ of a given finite-dimensional diffeological vector space $V$ with the appropriate Euclidean
space ($\matR^n$ endowed with its canonical scalar product) determines the decomposition of $V$ into the smooth direct sum of its maximal isotropic subspace
(its maximally non-standard part, as one could say) and its characteristic subspace (its maximal standard part).\footnote{In particular, the statement of
Proposition 3.10 in \cite{me2018} is actually true, although the proof as it appears therein, was incorrect. Now, mistakenly conceiving the characteristic
subspace as \emph{the} unique maximal subspace that is standard and splits off smoothly, was used in \cite{me2017b} to define the so-called characteristic
sub-bundle. In reality, the existence of such, not always guaranteed, as a smooth direct complement of \emph{the maximal isotropic sub-bundle} (which on the
other hand is always well-defined) must be imposed as a matter of assumption, together with the requirement that the gluing maps preserve them. Under such
assumption the results obtained in \cite{me2017b} concerning them still hold. A separate paper, in preparation, deals with these issues.}

Finally, in those cases when such dependence of the characteristic subspace on the choice of a coordinate system might be an obstacle, an alternative notion
can be used.

\begin{defn}
Let $V$ be a finite-dimensional diffeological vector space, and let $V_1$ be its maximal isotropic subspace. The \textbf{characteristic quotient} of $V$ is the
quotient vector space $V/V_1$ endowed with the quotient diffeology.
\end{defn}

Since the maximal isotropic subspace is uniquely determined by the vector space itself, the characteristic quotient is well-defined (and unlike the
characteristic subspace, it does not depend on introducing coordinates on $V$). We also have the following statement.

\begin{lemma}\label{characteristic-quotient-lem}
The characteristic quotient $V/V_1$ is diffeomorphic to the diffeological dual $V^*$ (and so to any characteristic subspace of $V$), and $V$ is diffeomorphic
to the direct sum $V/V_1\oplus V_1$. Furthermore, any pseudo-metric of $V$ descends to a well-defined scalar product on $V/V_1$.
\end{lemma}

\begin{proof}
This is a consequence of the existence of decompositions of $V$ into a smooth direct sum of its characteristic subspace and the maximal isotropic subspace,
and of Theorem \ref{map-onto-dual-thm}.
\end{proof}

Do notice that there is not a canonical diffeomorphism of $V$ with $V/V_1\oplus V_1$, for the same reasons for which in general there is not a canonical choice
of a characteristic subspace. Furthermore, characteristic quotients possess the following important property, whose analogue, as we will see in the next section,
does not hold for characteristic subspaces.

\begin{prop}
Let $V$ and $W$ be finite-dimensional diffeological vector spaces, and let $f:V\to W$ be a smooth linear map. Then $f$ descends to a well-defined linear (and
necessarily smooth) map $f_/:V/V_1\to W/W_1$ between their characteristic quotients.
\end{prop}

\begin{proof}
It is sufficient that the image $f(V_1)$ of the maximal isotropic subspace $V_1$ of $V$ be contained in the maximal isotropic subspace $W_1$ of $W$. Recall that
the maximal isotropic subspace can be described as the intersection of kernels of all smooth functions $W\to\matR$. Let $w\in f(V_1)$, so that $w=f(v)$ for some
$v\in V_1$, and let $h\in C^{\infty}(W,\matR)$. Then $h(w)=(h\circ f)(v)$, and $h\circ f$ is a smooth linear function on $V$. Since by assumption
$v\in V_1$, we must have $h(w)=(h\circ f)(v)=0$. Since $h$ is any element of $C^{\infty}(W,\matR)$, we conclude that $w$ belongs to $W_1$, and since $w$ is
any element of $f(V_1)$, we obtain that $f(V_1)\leqslant W_1$, as wanted.
\end{proof}

\subsection{Pseudo-metrics compatible with a linear map}

As was established in Theorem \ref{characteristic-is-well-defined-thm}, a diffeological vector space $V$ whose underlying space has a fixed identification with
some $\matR^n$, contains a uniquely defined characteristic subspace. We now show that these subspaces allow to establish, given a smooth linear map $V\to W$,
where $V$ and $W$ are diffeological vector spaces, whether these spaces can be endowed with pseudo-metrics such that the map is well-behaved with respect to those
pseudo-metrics. The precise meaning of being well-behaved here is quite standard.

\begin{defn}
Let $V$ and $W$ be finite-dimensional diffeological vector spaces endowed with pseudo-metrics $g_V$ and $g_W$ respectively, and let $f:V\to W$ be a smooth
linear map. We say that $g_V$ and $g_W$ are \textbf{compatible} with respect to $f$ (or simply, with $f$) if for any $v_1,v_2\in V$ we have
$g_V(v_1,v_2)=g_W(f(v_1),f(v_2))$.
\end{defn}

For the duration of this section we shall assume that the underlying vector spaces of $V$ and $W$ are $\matR^n$ and $\matR^m$, respectively. Let $V_0\leqslant V$
and $W_0\leqslant W$ be their characteristic subspaces, and let $V_1\leqslant V$ and $W_1\leqslant W$ be their maximal isotropic subspaces, so that
$V=V_0\oplus V_1$ and $W=W_0\oplus W_1$. Then the existence on $V$ and $W$ of pseudo-metrics compatible with a smooth linear map $f:V\to W$ is subject to a
number of necessary conditions.

\begin{thm} \emph{(Proposition 2.4, Lemmas 2.2, 2.3, 2.5 and Theorem 2.7 of \cite{me2017b})}
Let $f:V\to W$ be a smooth linear map such that $V$ and $W$ admit pseudo-metrics compatible with $f$. Then:
\begin{enumerate}
\item $\dim(V^*)\leqslant\dim(W^*)$;
\item $\mbox{Ker}(f)\leqslant V_1$;
\item The subset diffeology on $f(V_0)$ relative to its inclusion in $W$ is standard;
\item The subspace $f(V_0)$ splits off smoothly in $W$.
\end{enumerate}
\end{thm}

Notice that item 1 in the above theorem is actually a consequence of item 3 of same and of Corollary \ref{max-dim-of-standard-subspace-cor}, and it is the only
condition that depends only on the spaces $V$ and $W$ themselves, with no reference to a specific $f$. Observe also that $\mbox{Im}(f)$ is, by item 2, isomorphic
to $V_0\oplus V_1/\mbox{Ker}(f)$; however, the subset diffeology on $\mbox{Im}(f)$ relative to its inclusion into $W$ may \emph{a priori} be larger when the
diffeology it inherits from $V$ via the pushforward construction, \emph{i.e.} the direct sum diffeology relative to the standard one on $V_0$ and the quotient
one on $V_1/\mbox{Ker}(f)$. However, $f(V_0)$ taken alone inherits certain properties of characteristic subspaces.

It needs to be said that the proof of Corollary 2.6 in \cite{me2017b} is incorrect, since it used items 3 and 4 above and was based on the mistaken belief that the characteristic
subspace is the only maximal standard subspace splitting off smoothly. Furthermore, the statement itself of the corollary is wrong, since there are smooth maps
(in fact, diffeomorphisms) that do not preserve characteristic subspaces.

\begin{example}\label{characteristic-not-invt-under-diffeo-ex}
Let $V$ and $g$ be as in the Example \ref{characteristic-not-unique-ex}, and let $f:V\to V$ be the linear map given by $e_1\mapsto e_1$ and $e_2\mapsto e_1+e_2$.
Observe that $f$ is clearly compatible with $g$. To show that it is smooth, consider again the generic form of a plot $Q$ of $V$, which is
$$u\mapsto\left(\alpha_1(u)+\sum_if_i(u)|F_i(u)|,\alpha_2(u)\right)$$
for $u\in U$, the domain of definition of $Q$, and some usual smooth functions $\alpha_1,\alpha_2,f_i,F_i\in C^{\infty}(U,\matR)$. We have
$$(f\circ Q)(u)=\left(\alpha_1(u)+\alpha_2(u)+\sum_if_i(u)|F_i(u)|,\alpha_2(u)\right),$$
which again has the form of a generic plot of $V$. Hence, $f$ is smooth (and we can similarly prove that its inverse, acting by $e_1\mapsto e_1$,
$e_2\mapsto e_2-e_1$, is also smooth, showing that $f$ is moreover a diffeomorphism).
\end{example}

In particular, in Theorem 2.9 of \cite{me2017b} the requirement that $f(V_0)\leqslant W_0$, must be replaced by asking $f(V_0)$ to be standard and split off
smoothly, so that we have the following statements.

\begin{thm}
Let $V$ and $W$ be such that $\dim(V^*)\leqslant\dim(W^*)$, and let $f:V\to W$ be a smooth linear map. Then $V$ and $W$ admit pseudo-metrics compatible with $f$
if and only if $\mbox{Ker}(f)\leqslant V_1$, and $f(V_0)$ is standard and splits off smoothly.
\end{thm}

The proof of the statement is, however, word-for-word that given for the Theorem 2.9 of \cite{me2017b} (indeed, the proof therein relied on the existence of a
smooth decomposition into the direct sum of a standard subspace and the isotropic subspace, with no assumption of its uniqueness).

Finally, let us consider the dual map $f^*:W^*\to V^*$ and the compatibility of the induced pseudo-metrics. Let $V$ and $W$ be endowed with pseudo-metrics $g_V$
and $g_W$ respectively; we denote by $g_V^*$ and $g_W^*$ the induced pseudo-metrics on $V^*$ and $W^*$ (\emph{i.e.}, $g_V^*(v_1^*,v_2^*)=g_V(v_1,v_2)$, where
$v_i\in V$ is any such element that $v_i^*(\cdot)=g(v_i,\cdot)$, and likewise for $g_W^*$). These induced pseudo-metrics are said to be \textbf{compatible} with
$f^*$ if
$$g_V^*(f^*(w_1^*),f^*(w_2^*))=g_W^*(w_1^*,w_2^*)\,\,\,\mbox{ for all }w_1^*,w_2^*\in W.$$
Recall also that $V^*$ and $W^*$ are standard, so $g_V^*$ and $g_W^*$ are usual scalar products.

\begin{thm} \emph{(Theorem 2.11 of \cite{me2017b})}
Let $V$ and $W$ be such that $\dim(V^*)\leqslant\dim(W^*)$, and let $f:V\to W$ be a smooth linear map such that $V$ and $W$ can be endowed with pseudo-metrics
$g_V$ and $g_W$ compatible with $f$. Then the corresponding induced pseudo-metrics $g_W^*$ and $g_V^*$ are compatible with $f^*$ if and only if $f^*$ is a
diffeomorphism.
\end{thm}

\section{Diffeological algebras and modules}

The notions of a diffeological algebra, a diffeological module, and a smooth action of the former on the latter are obtained from the usual ones in the most
natural way, by imposing the requirement for all the maps (such as the operation maps and so on) involved to be smooth.

\subsection{Diffeological algebras}

As is the most natural, a \textbf{diffeological algebra} is an algebra $A$ endowed with a diffeology such that $A$ is also a diffeological vector space, and
the multiplication operation is a smooth map. It is trivial to show (Lemma 2.7 of \cite{me2019}) that any subalgebra of $A$ is again a diffeological algebra
for the subset diffeology. If $I$ is an ideal of $A$, the quotient $A/I$ is a diffeological algebra for the quotient diffeology (Lemma 2.8 of \cite{me2019}).

\begin{example} \emph{(Observation 3.2 of \cite{me2019})}
Let $A$ be the usual algebra of $2\times 2$ matrices with real coefficients. The smallest algebra diffeology on $A$ that contains the plot
$\matR\ni x\mapsto\left(\begin{array}{cc} 0 & 0 \\ 0 & |x| \end{array}\right)$ is precisely the diffeology where every plot locally has form,
$$x\mapsto\left(\begin{array}{cc} f_{11}(x)+\sum_ig_{11}^i(x)|h_{11}^i(x)| & f_{12}(x)+\sum_jg_{12}^j(x)|h_{12}^j(x)| \\
f_{21}(x)+\sum_kg_{21}^k(x)|h_{21}^k(x)| & f_{22}(x)+\sum_lg_{22}^l(x)|h_{22}^2(x)| \end{array}\right),$$
where $f_{ab},g_{ab}^c,h_{ab}^c$ are usual smooth functions. Notice that such typical functions as $\det$ and $\mbox{\emph{tr}}$ are not smooth on $A$ (Lemma 3.3 of
\cite{me2019}).
\end{example}

The definition of a diffeological module is also a natural one.

\begin{defn}
A diffeological vector space $E$ is a \textbf{diffeological module over} a diffeological algerba $A$ if there is a fixed homomorphism $c:A\to L^{\infty}(E,E)$
that is smooth for the diffeology on $A$ and the functional diffeology on $L^{\infty}(E,E)$.
\end{defn}

Let us consider a simple example of a non-standard diffeological module (this is essentially the remark in Section 3.2.2.3 of \cite{me2019}).

\begin{example}
Let $V$ again be $\matR^2$ endowed with the vector space diffeology generated by the plot $p:\matR\to V$ acting by $x\mapsto|x|e_1$, and let $A$ be the
algebra composed of all matrices of form $\left(\begin{array}{cc} a & 0 \\ 0 & 0 \end{array}\right)$ for $a\in\matR$, endowed with the algebra diffeology
generated by the plot $\matR\ni x\mapsto\left(\begin{array}{cc} |x| & 0 \\ 0 & 0 \end{array}\right)$. Then the usual action of $A$ on $V$ by left multiplication
is smooth. On the other hand, the only diffeology for which the action on $V$ of the algebra composed of all upper-triangular matrices (\emph{i.e.}, such that
the $(2,1)$th coefficient is zero) is smooth, is the standard diffeology (this is fully analogous to the Proposition 3.4 of \cite{me2019}).
\end{example}

The above example can also be used to illustrate the known fact (see also \cite{wu}) that, in contrast to the standard case, $c$ does not have to appear as an
element of $L(A,L(E,E))\cong A^*\otimes L(E,E)\cong A^*\otimes E^*\otimes E$, in the sense that it is by definition an element of the first of these three
spaces, but neither of the isomorphisms have to hold. Indeed, in the above example the diffeological vector space underlying the algebra $A$ is just $\matR$
endowed with the vector space diffeology generated by the absolute value function. The diffeological dual of this space, and therefore of $A$, is of course
trivial, and therefore so are the second and the third spaces above. But $c$ is obviously a non-zero action (to each matrix, it assigns the linear map that
consists in multiplication by the $(1,1)$th coefficient of the matrix).

\subsection{The tensor algebra of a diffeological vector space}

The tensor algebra of a diffeological vector space $V$ is the only instance of an infinite-dimensional diffeological vector space that appears in this paper.
It is a diffeological algebra for the tensor product diffeology (Lemma 4.1 of \cite{me2019}) and the (direct) sum diffeology (in particular, its decomposition as
$\bigoplus_{k=0}^{\infty}V^{\otimes k}$ is smooth by construction), so its subalgebras, ideals, and quotients all have natural
diffeological structures.

Such classic operators as the symmetrization operator
$$\mbox{Sym}:\underbrace{V\otimes\ldots\otimes V}_n\ni v_1\otimes\ldots\otimes v_n\mapsto\frac{1}{n!}\sum_{\sigma\in S_n}v_{\sigma(1)}\otimes\ldots\otimes
v_{\sigma(n)}\in\underbrace{V\otimes\ldots\otimes V}_n$$
and the antisymmetrization operator
$$\mbox{Alt}:\underbrace{V\otimes\ldots\otimes V}_n\ni v_1\otimes\ldots\otimes v_n\mapsto\frac{1}{n!}\sum_{\sigma\in S_n}\mbox{sgn}(\sigma)v_{\sigma(1)}\otimes
\ldots\otimes v_{\sigma(n)}\in\underbrace{V\otimes\ldots\otimes V}_n,$$
where $S_n$ stands for the group of permutations on $n$ elements, and the two operators are extended by linearity, are smooth as maps
$\underbrace{V\otimes\ldots\otimes V}_n\to\underbrace{V\otimes\ldots\otimes V}_n$. Furthermore, the following is true.

\begin{prop} \emph{(Lammas 2.10 and 2.11 of \cite{me2019})}
For both $\mathcal{S}_n(V)$, the space of all symmetric $n$-tensors on $V$, and $\mathcal{A}_n(V)$, the space of all antisymmetric $n$-tensors on $V$, the
subset diffeology relative to the inclusion into $T(V)$ and the pushforward of the diffeology of $T(V)$ by, respectively, $\mbox{Sym}$ and $\mbox{Alt}$,
coincide.
\end{prop}

\subsection{The exterior algebra}

For any finite-dimensional diffeological vector space $V$, the algebra
$$\bigwedge^nV^*=\mbox{Alt}(\underbrace{V^*\otimes\ldots\otimes V^*}_n)$$
of all antisymmetric covariant $n$-tensors is a standard space (since $V^*$ is so). In particular, the usual exterior product
$$\wedge:\bigwedge^kV^*\times\bigwedge^lV^*\to\bigwedge^{k+l}V^*$$
is trivially smooth. If $V$ has a fixed basis, its exterior algebra coincides with that of its characteristic subspace. In either case, it is diffeomorphic
to that of its \emph{characteristic quotient} (the quotient of $V$ by its maximal isotropic subspace).

There is of course also the contravariant version of the exterior algebra, which (unlike the standard case, and precisely because in general $V^*$ is not
isomorphic to $V$ itself) needs to be treated separately. We denote this contravariant exterior algebra by
$$\bigwedge_*^nV=\mbox{Alt}(\underbrace{V\otimes\ldots\otimes V}_n).$$
In this case the smoothness of the exterior product is not automatically guaranteed, but it follows from the smoothness of the antisymmetrization operator.

\subsection{Diffeological Clifford algebras and actions}

Let $V$ be a finite-dimensional diffeological vector space, and let $q$ be a smooth symmetric bilinear form on $V$. As in the standard case, the Clifford
algebra associated to $V$ and $q$ is the quotient $T(V)/I(V)$ of the tensor algebra $T(V)$ by its ideal $I(V)$ generated by all elements of form
$v\otimes w+w\otimes v+4q(v,w)$, where $v,w\in V$, the ideal $I(V)$ has the subset diffeology, and the quotient is endowed with the quotient diffeology. (Notice
that this yields a contravariant version of a Clifford algebra. One may also consider its covariant version, which in general is different --- with the
underlying vector spaces being distinct --- from the contravariant version).

\begin{example}\label{contravariant-clifford-algebra-ex}
Let $V$ be as in the Example \ref{characteristic-not-unique-ex}, that is, $\matR^2$ with the diffeology given by the plot $p:\matR\ni x\mapsto |x|e_1$, and
let $q$ be the pseudo-metric $g$ described in the same example, \emph{i.e.}, given by the matrix $\left(\begin{array}{cc} 0 & 0 \\ 0 & 1 \end{array}\right)$.
The corresponding Clifford algebra $\cl(V,g)$ is therefore determined by the relations $e_1\otimes e_1=0$, $e_1\otimes e_2=-e_2\otimes e_1$, and
$e_2\otimes e_2=-2$. Its underlying vector space, of dimension $4$, is generated by $1$, $e_1$, $e_2$, and $e_1\cdot e_2$ (where $\cdot$ denotes the product
in the Clifford algebra), and so is diffeomorphic to $\matR^4$ with the vector space diffeology generated by the plot $\matR\ni x\mapsto(0,|x|,0,0)$ and
$\matR\ni x\mapsto(0,0,0,|x|)$. Observe that its maximal isotropic subspace is the ideal of $\cl(V,q)$ generated by the maximal isotropic subspace,
$\mbox{Span}(e_1)$, of the initial vector space $V$.
\end{example}

In general, it is a matter of standard reasoning to observe that if $V=V_0\oplus V_1$ is a decomposition of a (finite-dimensional) diffeological vector space $V$
endowed with a pseudo-metric $g$, into the smooth direct sum of a characteristic subspace and the maximal isotropic subspace, $g_0$ is the restriction of $g$
onto $V_0$, $g_/$ is the pushdown of $g$ onto the characteristic quotient $V/V_1$, and $g^*$ is the induced pseudo-metric on $V^*$, then of course
$\cl(V_0,g_0)\cong\cl(V/V_1,g_/)\cong\cl(V^*,g^*)$, and for any basis $v_1',\ldots,v_{n-k}'$ of $V_1$ we have
$$\cl(V,g)=\cl(V_0,g_0)\oplus\left(\bigoplus_{i=1}^{n-k}v_i'\cl(V_0,g_0)\right),$$
where the internal product on the second factor is always zero, while products by elements of $\cl(V_0,g_0)$ keep each summand of form $v_i'\cl(V_0,g_0)$
invariant. Since also any element of $\cl(V_0,g_0)$ anticommutes with any $v_i'$, these products are essentially determined by the multiplication in
$\cl(V_0,g_0)$, and so the entire $\cl(V,g)$ is essentially determined by $\cl(V_0,g_0)$ (or, alternatively, by $\cl(V/V_1,g_/)$ or $\cl(V^*,g^*)$) and the
dimension of its maximal isotropic subspace (which is $\dim(V)-\dim(V^*)$). Moreover, the following is true.

\begin{lemma}
For any choice of a characteristic subspace, $\cl(V_0,g_0)$ splits off smoothly in $\cl(V,g)$.
\end{lemma}

\begin{proof}
Recall first the diffeology of $T(V)$ is by definition the direct sum diffeology corresponding to its presentation as $\bigoplus_{k=0}^{\infty}V^{\otimes k}$.
This in particular means that any nonconstant plot of it is a sum of a finite number of plots of finite tensor degrees of $V$. Since also the smoothness of a
direct sum is respected under the passage to a quotient diffeology, it suffices to show that for any smooth decomposition of $V$ as $V=V_1\oplus V_2$, the direct
sum decomposition of $V\otimes V$ as $(V_1\otimes V)\oplus(V_2\otimes V)$ is smooth, which immediately follows from Lemma \ref{tensor-prod-is-distributive-lem}.
\end{proof}

In particular, the covariant Clifford algebra $\cl(V^*,g^*)$ is diffeomorphic to a subalgebra of the contravariant Clifford algebra $\cl(V,g)$ that, as a vector
subspace, splits off smoothly. A more standard direct sum representation of a Clifford algebra such as its $\mathbb{Z}_2$-grading
$$\cl(V,g)=\cl(V,g)^0\oplus\cl(V,g)^1$$
is also a smooth direct sum decomposition (Proposition 4.8 of \cite{me2019}). Likewise, the standard filtration of $\cl(V,g)$ by $\cl^k(V,g)/\cl^{k-1}(V,g)$
has the expected smoothness properties in that the subset diffeology on $\cl^k(V,g)$ (this being, as usual, the image of the restriction $\pi^k$ to
$T^k(V)=\bigoplus_{r=0}^kV^{\otimes r}$ of the natural projection $\pi:T(V)\to\cl(V,g)$) relative to its inclusion $\cl^k(V,g)\subset\cl(V,g)$ coincides
with the pushforward of the subset diffeology on $T^k(V)\subset T(V)$ by the projection map $\pi^k$ (Lemma 4.9 of \cite{me2019}), and, with respect to the
quotient diffeology on $\cl^k(V,g)/\cl^{k-1}(V,g)$, the usual isomorphism $\cl^k(V,g)/\cl^{k-1}(V,g)\to\bigwedge_*^kV$ is a smooth map with a smooth
inverse.

Since the case of the covariant Clifford algebra $\cl(V^*,g^*)$ and the covariant exterior algebra $\bigwedge V^*$ coincides with the standard one,
$\bigwedge V^*$ is a Clifford moodule via the standard action of $\cl(V^*,g^*)$ (that is, the action $c:\cl(V^*,g^*)\to\mbox{End}(\bigwedge V^*)$ determined by
$$c(v^*)=\varepsilon(v^*)-i(v^*)\,\,\mbox{ for all }\,\,v^*\in V^*$$
with $\varepsilon(v^*)$ acting on $\bigwedge V^*$ by the left exterior product ($\varepsilon(v^*)(\alpha)=v^*\wedge\alpha$ for all $\alpha\in\bigwedge V^*$) and
$i(v^*)$ being the adjoint of $\varepsilon(v^*)$ via $g^*$, \emph{i.e.}, acting by
$$i(v^*)(w_1\wedge\ldots\wedge w_l)=\sum_{i=1}^l(-1)^{i+1}w_1\wedge\ldots\wedge g^*(v^*,w_i)\wedge\ldots\wedge w_l.$$ This can be carried over word-for-word to
the contravariant case (which would in fact be the case of a Clifford algebra associated to a bilinear form with some degree of degeneracy), in which case we
observe that the action is smooth (Section 4.2.2 of \cite{me2019}).

\begin{example}
Let $V$ and $g$ be as in the Example \ref{contravariant-clifford-algebra-ex}. Since we would have $i(e_1)=0$ (this would obviously be the case for any $V$ and
for any element of its maximal isotropic subspace),
$$c(e_1)=\varepsilon(e_1)=e_1\wedge,$$
\emph{i.e.} it is the map acting on the basis $1,e_1,e_2,e_1\wedge e_2$ by $1\mapsto e_1$, $e_1\mapsto 0$, $e_2\mapsto e_1\wedge e_2$, and
$e_1\wedge e_2\mapsto 0$. For $e_2$, we would have that $i(e_2)$ is given by
$i(e_2)(1)=0$, $i(e_2)(e_1)=0$, $i(e_2)(e_2)=1$, and $i(e_2)(e_1\wedge e_2)=-e_1$, while $\varepsilon(e_2)=e_2\wedge$ acts by $\varepsilon(e_2)(1)=e_2$,
$\varepsilon(e_2)(e_1)=-e_1\wedge e_2$, and $\varepsilon(e_2)(e_2)=\varepsilon(e_2)(e_1\wedge e_2)=0$. Thus, $c(e_2)$ acts by $c(e_2)(1)=e_2$,
$c(e_2)(e_1)=-e_1\wedge e_2$, $c(e_2)(e_2)=-1$, and $c(e_2)(e_1\wedge e_2)=e_1$. Observe also that the classic isomorphism $\sigma:\cl(V,g)\to\bigwedge_*V$
given by $\sigma(x)=c(x)(1)$ for all $x\in\cl(V,g)$, is indeed an isomorphism, since we have $c(1)(1)=1$, $c(e_1)(1)=e_1$, $c(e_2)(1)=e_2$, and (accordingly)
$c(e_1\cdot e_2)(1)=e_1\wedge e_2$. It is also smooth (Proposition 4.3 of \cite{me2019}; it was actually stated, in a somewhat tautological way, for scalar
products, but the proof as given therein works for any smooth bilinear form).
\end{example}

\vspace{1cm}

\noindent University of Pisa \\
Department of Mathematics \\
Via F. Buonarroti 1C\\
56127 PISA -- Italy\\
\ \\
ekaterina.pervova@unipi.it\\

\end{document}